\documentclass{article}
\usepackage{amsmath}
\usepackage{amsthm}
\usepackage{amssymb}
\usepackage{array}
\usepackage{bbm}
\usepackage{booktabs}
\usepackage{caption}
\usepackage{enumerate}
\usepackage{epstopdf}
\usepackage{fancyhdr}
\usepackage{geometry}
\usepackage{graphicx}
\usepackage[colorlinks,linkcolor=red,anchorcolor=green,citecolor=blue]{hyperref}
\usepackage{indentfirst}
\usepackage[utf8]{inputenc}
\usepackage{listings}
\usepackage{mathrsfs}   
\usepackage{multirow}
\usepackage{pgfplots} 
\usepackage{setspace}
\usepackage{stmaryrd}
\usepackage{tabularx}
    \newcolumntype{C}[1]{>{\centering\let\newline\\\arraybackslash\hspace{0pt}}m{#1}}
\usepackage{threeparttablex}
\usepackage{tikz} 
    \usetikzlibrary{arrows, decorations.pathmorphing, backgrounds, positioning, fit, petri, automata}
    \usetikzlibrary{decorations.markings, intersections, arrows.meta, positioning, calc, math, matrix}
    \usetikzlibrary{patterns}
    \usetikzlibrary{shadings}
    \usetikzlibrary{3d}
\usepackage{tikz-cd}
\usepackage{xcolor}
\usepackage{xypic}      
\title{ }
\author{ }
\date{}

\begin{document}
\theoremstyle{definition} 
    \newtheorem{fact}{Fact}
    \newtheorem{pro}{Question}
    \newtheorem{thm}{Theorem}[section]
    \newtheorem*{mainthm}{Main Theorem}
    \newtheorem{thmIntro}{Theorem}
    \renewcommand{\thethmIntro}{\Alph{thmIntro}}
    \newtheorem{coroIntro}[thmIntro]{Corollary}
    \newtheorem{lem}[thm]{Lemma}
    \newtheorem{coro}[thm]{Corollary}
    \newtheorem{ex}[thm]{Example}
    \newtheorem{defi}[thm]{Definition}
    \newtheorem{prop}[thm]{Proposition}
    \newtheorem{property}[thm]{Property}
    \newtheorem{rmk}[thm]{Remark}
    \newtheorem*{thme}{THEOREM}

    \newenvironment{thmrewrite}[1]
    {\renewcommand{\thethm}{\ref{#1} (c)$'$}
     \addtocounter{thm}{-1}
     \begin{thm}}
    {\end{thm}}
    \newenvironment{thmrewrite2}[1]
    {\renewcommand{\thethm}{\ref{#1}$'$}
     \addtocounter{thm}{-1}
     \begin{thm}}
    {\end{thm}}

    \newcommand{\vphi}[0]{\varphi}
    \newcommand{\upst}[0]{^{*}}
    \newcommand{\dwst}[0]{_{*}}
    \newcommand{\dd}[1]{\textrm{d}#1}
    \newcommand{\metr}{\textrm{g}}
    \newcommand{\Metr}{\textrm{G}}
    \newcommand{\rr}{\textrm{R}}
\newcommand{\rmnum}[1]{\romannumeral #1}
\newcommand{\mR}{\mathbb{R}}
\newcommand{\mC}{\mathbb{C}}
\newcommand{\mH}{\mathbb{H}}
\newcommand{\mD}{\mathbb{D}}
\newcommand{\mS}{\mathbb{S}}
\newcommand{\mZ}{\mathbb{Z}}
\newcommand{\z}{\zeta}
\newcommand{\tphi}{\tilde{\phi}}
\newcommand{\cp}{\mathbb{CP}^1}
    \newcommand{\sm}{\mathcal{M}}

    \newcommand{\warp}{\times_{f}}
    \newcommand{\wrp}[1]{\times_{#1}}
    \newcommand{\warppro}{M\warp N}
    \newcommand{\DF}{F^{'}}
    \newcommand{\mHi}[1]{\mH^n_{(#1)}}
    \newcommand{\gi}[1]{\text{g}_{(#1)}}
    \newcommand{\uppri}{^{'}}
    \newcommand{\upcir}{^{\circ}}
    \newcommand{\inpro}[2]{\langle#1,#2\rangle}
    \newcommand{\diff}{\text{d}}
    \newcommand{\pp}[1]{\frac{\partial}{\partial #1}}
    \newcommand{\tg}[2]{\tilde{\metr}\left(#1,#2\right)}
    \newcommand{\tn}[2]{\tilde{\nabla}_{#1}#2}
    \newcommand{\tmf}{\tilde{\mathcal{F}}}
    \newcommand{\codim}{\operatorname{codim}}

    \newcommand{\F}{\mathcal{F}}
    \newcommand{\nsb}[1]{\text{NS}(#1)}
    \newcommand{\myeqref}[1]{\text{(}\ref{#1}\text{)}}

    \newcommand{\inv}[1]{\operatorname{Inv}_0(#1)}
    \newcommand{\invre}[1]{\operatorname{Inv}^{-}_0(#1)}

    \newcommand{\homeo}[1]{\operatorname{Heomo}\left(#1\right)}
    \newcommand{\homeoori}[1]{\operatorname{Heomo}^{+}\left(#1\right)}
    \newcommand{\isotoid}[1]{\operatorname{Heomo}_0\left(#1\right)}
    \newcommand{\MCG}[1]{\operatorname{Mod}\left(#1\right)}
    \newcommand{\eMCG}[1]{\operatorname{Mod}^{\pm}\left(#1\right)}

    \newtheorem{thmintro}{Theorem}
    \renewcommand{\thethmintro}{\Alph{thmintro}}

    \newtheorem{corointro}[thmintro]{Corollary}

\title{On the Classification of Isoparametric Hypersurfaces with Constant Principal Curvatures in Compact 3-Manifolds}

    \author{Minghao Li, Ling Yang}

    \renewcommand{\thefootnote}{}
    \footnotetext{\textit{Email addresses}: mhli19@fudan.edu.cn (Minghao Li), yanglingfd@fudan.edu.cn (Ling Yang).}

	\maketitle

	\renewcommand{\proofname}{\bf Proof.}

\begin{abstract}
    Establishing detailed relationships between transnormal systems of different types and their behaviors under covering maps, this paper presents a classification of transnormal systems on compact 3-manifolds in the sense of equivalence. For CPC transnormal systems, we show that the ambient manifolds must be locally isometric to one of six standard geometries up to equivalence. We also find some equivalence classes containing no CPC transnormal system, highlighting a critical distinction between isoparametric foliations and CPC transnormal systems, which has not been previously addressed in the literature.
\end{abstract}
\tableofcontents
\section{Introduction}
    A smooth function $f$ on a Riemannian manifold $(M,g)$ is called {\it transnormal} if there is a smooth function $b$ such that $|\nabla f|^2=b(f)$. If, in addition, $\Delta f=a(f)$ for another smooth function $a$, then $f$ is called {\it isoparametric}. Each regular level hypersurface of an isoparametric function is called an {\it isoparametric hypersurface}. The classification of isoparametric hypersurfaces has been extensively studied in various spaces. In Euclidean and hyperbolic spaces, the classification was completed by Cartan in \cite{Cartan1,Cartan2}. In the standard spheres, the classification, which is significantly more challenging, was initiated by Cartan and eventually completed through contributions from many works, such as \cite{ozeki1975some,ferus1981cliffordalgebren,takagi1972principal,munzner1980isoparametric,munzner1981isoparametric,abresch1983isoparametric,dorfmeister1985isoparametric,tang1991fourdistinct,fang1999topology,cecil2007isoparametric,immervoll2008classification,chi2011isoparametric,chi2013isoparametric,miyaoka2013isoparametric,miyaoka2016errata,chi2020isoparametric}, with further surveys available in \cite{chiSurvey}. Further classifications have been explored in spaces such as complex projective spaces $\mC P^n$ \cite{wang1982isoparametric,kimura1986real,park1989isoparametric,dominguez2016isoparametric}, complex hyperbolic spaces $\mC\mH^n$ \cite{berndt1989real,berndt2006real,berndt2007real,diaz2012inhomogeneous,diaz2017isoparametric}, compact symmetric spaces \cite{murphy2012curvature}, and the product space $\mS^2 \times \mS^2$ \cite{urbano2019hypersurfaces}.
    \par    
    For general Riemannian manifolds, thanks to the tubular regularity of transnormal functions in \cite{wang1987isoparametric}, an arbitrary transnormal function induces an embedded transnormal system of codimension one. A {\it transnormal system} $\F$ on a complete Riemannian manifold $M$ is a partition of $M$ into connected submanifolds, called {\it foils}, so that any geodesic in $M$ intersects these foils orthogonally at all or none of its points (see \cite{bolton1973transnormal}). $\F$ is {\it embedded} if all foils are embedded; the {\it codimension} of $\F$ is defined as the codimension of the foil with the greatest dimension. Conversely, it is shown that any embedded transnormal system of codimension one can be induced by a transnormal function, and in the case of compact manifolds, it can further be realized as an isoparametric function. Precisely, the following results hold:
    \begin{thm}[\cite{miyaoka2013transnormal,li2024}]
        Suppose $\F$ is an embedded transnormal system of codimension one on a connected complete manifold $M$, then $M$ admits a transnormal function $f$ such that the transnormal system $\F_f$ induced by $f$ coincides with $\F$. 
    \end{thm}
    \par    
    \begin{thm}[\cite{li2024}]\label{thm: intro: four equivlent conditions}
        For each compact manifold $M$, the following statements are equivalent:
        \begin{enumerate}[(a)]
            \item $M$ admits a linear double disk bundle decomposition.
            \item $M$ can be endowed with a Riemannian metric so that it admits an embedded transnormal system of codimension one.
            \item $M$ can be endowed with a Riemannian metric so that it admits a transnormal function.
            \item $M$ can be endowed with a Riemannian metric so that it admits an isoparametric function. 
        \end{enumerate}
    \end{thm}
    \par    
    Moreover, in the equivalence of Theorem \ref{thm: intro: four equivlent conditions}, the foils of the transnormal system correspond precisely to the base spaces of the disk bundles and to the connected components of their subbundles with hypersphere fibers, where an SR-foil or an S-foil, if it exists, is exactly a base space.
    \par    
    The next natural step is to investigate the classification of transnormal functions and manifolds admitting them. In contrast to the classification of isoparametric functions (or transnormal functions) on a fixed Riemannian manifold, the classification in general Riemannian manifolds cannot, in general, be carried out up to isometry. Instead, we develop the classification in the following equivalent sense: 
    \begin{defi}\label{def: equivalence of tran sys}
        Let $\F$ and $\F\uppri$ be transnormal systems on Riemannian manifolds $M$ and $M\uppri$, respectively. $\F$ and $\F\uppri$ are {\it equivalent} if there exists a diffeomorphism $\phi:M\to M\uppri$ such that $\phi(F)\in\F\uppri$ for any $F\in\F$.  
    \end{defi}
    \par    
    For simply connected compact Riemannian manifolds, the classification is well understood in low dimensions. In dimension four, any Riemannian manifold admitting a transnormal function is diffeomorphic to $\mS^4$, $\mC\mathbb{P}^2$, $\mS^2\times\mS^2$, or $\mC\mathbb{P}^2\#\pm \mC\mathbb{P}^2$, as established in \cite{ge2015differentiable}. The five-dimensional case was later settled in \cite{devito2023manifolds}, identifying the diffeomorphism types as $\mS^5$, the Wu manifold $\operatorname{SU}(3)/\operatorname{SO}(3)$, $\mS^3\times\mS^2$, and the unique twisted $\mS^3$-bundle over $\mS^2$. In dimension three, the problem is trivial as a consequence of the Poincar\'{e} conjecture. However, much less is known when the fundamental group is nontrivial. It is important to note that such cases cannot be resolved by lifting to the universal cover: transnormal systems and functions depend on the Riemannian metric, and the isometric covering transformation group is not invariant under diffeomorphisms. A more detailed discussion of this issue is given in Remark \ref{remark: why lift directly to universal is bad}. 
    \par   
    In this paper, we initiate the study of the classification of compact Riemannian manifolds that admit transnormal functions and systems, extending beyond the simply connected case. In particular, the classification of transnormal systems on compact 3-manifolds is carried out based on their types. More precisely, according to the base spaces of the disk bundles in Theorem \ref{thm: intro: four equivlent conditions}, they fall into the following four types:
    \begin{itemize}
        \item {\it toric type}: both base spaces have codimension one, and both disk bundles are trivial; 
        \item {\it Klein-bottled type}: both base spaces have codimension one, and both disk bundles are non-trivial; 
        \item {\it spherical type}: both base spaces have codimension greater than one; 
        \item {\it real-projective type}: one base space has codimension one, while the other has codimension greater than one.
    \end{itemize}
    \par    
    We provide a simple necessary and sufficient condition for determining the equivalence of transnormal systems of toric type and Klein-bottled type in the sense of Definition \ref{def: equivalence of tran sys}. For spherical type and real-projective type, their classification up to equivalence is given in Tables \ref{table: intro: Classification of spherical type in closed 3d} and \ref{table: intro: Classification of real-projective type in closed 3d}, respectively. 
    \begin{table}[!htbp]
        \centering \small
        \caption{Equivalence classes of transnormal systems of spherical type on compact 3-manifolds.}
        \label{table: intro: Classification of spherical type in closed 3d}
        \renewcommand\arraystretch{1.3}
        \tabcolsep=0.5cm
        \scalebox{1}{
        \begin{tabular}{C{2cm}|C{2cm}|C{1.5cm}|C{2.5cm}|C{2cm}} 
            \toprule[1pt]
            Manifold $M$  & DR-foils & S-foils & Cohomogeneity-One Action inducing $\F$ & Orientable ($M$) \\
            \cline{1-5}\cline{1-5}
            $\mS^3$ & $2$-sphere & Point & $\operatorname{SO}(3,\mR)$ & \multirow{2}{*}{Yes}\\
            \cline{1-4}
            Lens Space & Torus & \multirow{2}{*}{Circle} & $\mS^1\times\mS^1$ & \\
            \cline{1-2}\cline{4-5}
            $\mS^1\times_{\text{\tiny tw}} \mS^2$ & Klein bottle && None & No \\
            \bottomrule[1pt]
        \end{tabular}
        }
    \end{table}
    \begin{table}[!htbp]
        \centering \small
        \caption{Equivalence classes of transnormal systems of real-projective type on compact 3-manifolds. }
        \label{table: intro: Classification of real-projective type in closed 3d}
        \renewcommand\arraystretch{1.3}
        \tabcolsep=0.5cm
        \begin{tabular}{C{2cm}|C{2cm}|C{1.8cm}|C{2.3cm}|C{2cm}}
            \toprule[1pt]
            Manifold $M$ & DR-foils & SR-foil &  Cohomogeneity-One Action inducing $\F$ & Orientable ($M$) \\
            \hline
            $\mR P^3$  & 2-sphere & $\mR P^2$ &  $\operatorname{SO}(3,\mR)$ & Yes\\
            \hline
            $\mS^1\times\mR P^2$  & \multirow{3}{*}{Torus} & \multirow{2}{*}{Torus} &  \multirow{2}{*}{$\mS^1\times\mS^1$}& \multirow{2}{*}{No}\\
            \cline{1-1}
            $\mS^1\times_{\text{\tiny tw}} \mS^2$ & & & &\\
            \cline{1-1}\cline{3-5}
            Klein Spaces & & \multirow{2}{*}{Klein bottle} & \multirow{2}{*}{None} & Yes\\
            \cline{1-2}\cline{5-5}
            $\mS^1\times\mR P^2$ & Klein bottle & & & No\\
            \bottomrule[1pt]
        \end{tabular}
    \end{table}
    The classification of the former two types is not explicitly listed due to the large number of equivalence classes. Furthermore, Table \ref{table: intro: closed 3-manifold} presents a summary of the classification for all four types, organized by the topological types of the ambient manifolds. This result also extends the classification of cohomogeneity one actions on compact 3-manifolds in \cite{neumann19683}, as such actions necessarily induce a transnormal system of codimension one. 
    \begin{table}[!htbp]
        \centering \small
        \caption{Compact 3-manifolds admitting transnormal systems}
        \label{table: intro: closed 3-manifold}
        \renewcommand\arraystretch{1.3}
        \tabcolsep=0.5cm
        \scalebox{0.95}{
            \begin{tabular}{C{4.5cm}|C{1.5cm}|C{2.2cm}|C{1.3cm}|C{2.0cm}}
                \toprule[1pt]
                \multirow{2}{*}{Topology of $M$} & \multicolumn{4}{c}{Types of transnormal systems admitted} \\
                \cline{2-5}
                 & spherical & real-projective & toric & Klein-bottled \\
                \cline{1-5}
                $\mS^3$ ($=L(1,0)$) & $\surd$ & & & \\
                \cline{1-5}
                $\mS^2\times\mS^1$  & \multirow{2}{*}{$\surd$} & & \multirow{2}{*}{$\surd$} &  \\
                ($=L(0,1)=M(1,0)=\mathcal{M}_{id_{\mS^2}}$) &&&&\\
                \cline{1-5}
                $\mS^1\times_{\text{\tiny tw}}\mS^2(=\mathcal{M}_{\sigma_{\mS^2}})$ & $\surd$ & $\surd$ & $\surd$ &  \\
                \cline{1-5}
                $\mR P^3$ ($=L(2,1)$) & $\surd$ & $\surd$ & &  \\
                \cline{1-5}
                $\mR P^2\times\mS^1(=\mathcal{M}_{id_{\mR P^2}})$ &  & $\surd$ & $\surd$ &  \\
                \cline{1-5}
                $\mR P^3\#\mR P^3$ &  & \multirow{2}{*}{$\surd$} &  & \multirow{2}{*}{$\surd$}  \\
                ($= M(0,1)= \mathcal{K}_{\sigma_{\scriptscriptstyle\mS^2};\sigma_{\scriptscriptstyle\mS^2}}$) &&&&\\
                \cline{1-5}
                Other lens spaces $L(p,q)$ & $\surd$ &  &  &  \\
                \cline{1-5}
                Other Klein spaces $M(p,q)$ &  & $\surd$ & &  \\
                \cline{1-5}
                Other $\mathcal{M}_{\tilde{\tau}}$ &  &  & $\surd$ &  \\
                \cline{1-5}
                Other $\mathcal{K}_{\tilde{\sigma}_1;\tilde{\sigma}_2}$ &  &  & & $\surd$  \\
                \bottomrule[1pt]
            \end{tabular}
        }
    \end{table}
    \par    
    Several aspects in these tables deserve further explanation:
    \begin{itemize}
        \item The Klein spaces $M(p,q)$ refer to a class of compact, orientable, irreducible 3-manifolds, in which a Klein bottle can be embedded, where $p$ and $q$ are coprime. For its precise definition, see Definition \ref{def: Klein space} or \cite{kim1978some,kim1981involutions}. The space $\mathcal{M}_{\tilde{\tau}}$ is the mapping torus induced by $\tilde{\tau}$. The space $\mathcal{K}_{\tilde{\sigma}_1;\tilde{\sigma}_2}$ is defined as \myeqref{equ: def of K_sigma12}. 
        \item ``None cohomogeneity-one action inducing $\F$'' in Table \ref{table: intro: Classification of spherical type in closed 3d} and Table \ref{table: intro: Classification of real-projective type in closed 3d} means that no transnormal system in this equivalence class can be induced by a cohomogeneity-one action. 
        \item Some spaces with simple structures are contained in ``other $\mathcal{M}_{\tilde{\tau}}$''. For example, when $\tilde{\tau}$ is the identity of the torus, $\mathcal{M}_{\tilde{\tau}}$ is homeomorphic to $\mS^1\times\mS^1\times\mS^1$.
        \item There may exist not only transnormal systems of different types (as shown in Table \ref{table: intro: closed 3-manifold}), but also multiple transnormal systems of the same type on a given manifold. For instance, the standard 3-sphere admits a transnormal system of spherical type whose DR-foils are all homeomorphic to the 2-sphere, as well as another whose DR-foils are all homeomorphic to the torus; for $\mR P^2\times\mS^1$, it admits both a transnormal system of real-projective type whose DR-foils are all homeomorphic to the torus and another  whose DR-foils are all homeomorphic to the Klein bottle. 
        \item Since many Klein spaces are also lens spaces (see \cite{kim1978some}), some 3-manifolds simultaneously belong to both "Other lens spaces" and "Other Klein spaces" in Table \ref{table: intro: closed 3-manifold}, such as $M(1,q)$ for odd $q$. Likewise, some 3-manifolds simultaneously belong to both "Other $\mathcal{M}_{\tilde{\tau}}$" and "Other $\mathcal{K}_{\tilde{\sigma}_1;\tilde{\sigma}_2}$" (for example, $\mathcal{K}_{\sigma_{\mathbf{K}};\sigma_{\mathbf{K}}}$ is also a mapping torus, where $\sigma_{\mathbf{K}}$ is a fixed-point-free involution on the Klein bottle). Other spaces in Table \ref{table: intro: closed 3-manifold} are all topologically distinct. This follows from the fact that, aside from the first six cases listed in Table \ref{table: intro: closed 3-manifold}, all other lens spaces and Klein spaces have finite fundamental groups with more than two elements, while all other $\mathcal{M}_{\tilde{\tau}}$ and other $\mathcal{K}_{\tilde{\sigma}_1;\tilde{\sigma}_2}$ are infinitely covered by $\mR^3$. 
    \end{itemize}
    \par    
    Since the classification is carried out in the sense of equivalence, a natural question arises: does each equivalence class admit a representative with desired geometric properties? By `desired', we mean the following conditions on the representative: 
    \begin{enumerate}
        \item[(A)] Each foil has constant mean curvature.
        \item[(B)] Each foil has constant principal curvatures.
        \item[(C)] The ambient manifold is locally isometric to a standard space.
        \item[(D)] Conditions (A) and (C) hold simultaneously. 
        \item[(E)] Conditions (B) and (C) hold simultaneously. 
    \end{enumerate}
    \par    
    Note that for transnormal systems induced by cohomogeneity one actions, all above conditions are clearly satisfied. Due to Theorem \ref{thm: intro: four equivlent conditions}, any transnormal system on a compact manifold is equivalent to an isoparametric foliation, which satisfies condition (A). To express condition (B), we define:
    \begin{defi}
        A transnormal system $\F$ on a Riemannian manifold is called a {\it CPC transnormal system} if each foil in $\F$ has constant principal curvatures.  
    \end{defi}
    \par   
    Some properties of CPC transnormal systems of codimension one are studied in \cite{ge2014geometry}. In the remainder of this paper, we focus on the equivalence classes of transnormal systems on orientable manifolds, and obtain the following conclusion:
    \begin{thm}\label{thm: intro: CPC has standard}
        Suppose $M$ is an orientable compact Riemannian 3-manifold, and $\F$ is an embedded CPC transnormal system of codimension one. Then, up to equivalence within CPC transnormal systems,
        $M$ is locally isometric to $\mS^3$, $\mS^2\times\mR$, $\mR^3$, $\mH^2\times\mR$, Nil geometry, or Sol geometry. 
    \end{thm}
    \par   
    As a consequence, we see that if condition (B) holds, then a refined choice of representative can ensure that (C) and (E) hold as well, and the standard spaces involved are precisely six of the eight Thurston's geometries. However, somewhat surprisingly, condition (B) does not always hold:
    \begin{thm}\label{thm: intro: no CPC in some equivalence class}
        There exist embedded transnormal systems of codimension one on compact Riemannian 3-manifolds such that their equivalence classes contain no CPC transnormal systems.
    \end{thm}
    \par    
    Additionally, the equivalence classes in Theorem \ref{thm: intro: no CPC in some equivalence class} sometimes do not even contain any transnormal system satisfying condition (D), as elaborated in Example \ref{example: mapping torus of pseudo-anosov of hyperbolic surface}. This suggests that, compared to cohomogeneity-one actions, transnormal systems impose significantly weaker topological and geometric constraints on the manifolds. On the other hand, while the topological equivalence between isoparametric foliations and transnormal systems has already been established in Theorem \ref{thm: intro: four equivlent conditions}, we now identify a critical distinction between isoparametric foliations and CPC transnormal systems in dimension three. This aspect has not been addressed in any dimension in the existing literature.

\section{Preliminaries}
    Let $\F$ be an embedded transnormal system of codimension one on a complete Riemannian manifold $M$. We will provide some basic definitions and discuss some global properties, which were established in our previous work \cite{li2024}. 
    \begin{defi}
        A foil $L$ of greatest dimension in $\F$ is a {\it regular} foil. We call $L$ a {\it double-sided regular foil} (abbreviated as DR-foil) if the normal sphere bundle of $L$ is disconnected; otherwise, $L$ is a {\it single-sided regular foil} (abbreviated as SR-foil). Consistently, A foil of higher codimension is called a {\it singular foil} (abbreviated as S-foil).
    \end{defi}
    \par    
    For any foil $L\in\F$, fixing an unit normal vector $V_0$ to $L$, we denote by $B(L)$ the connected component of its normal sphere bundle $\nsb{L}$ containing $V_0$. Let 
    \[
        T:=\sup\left\{t\in(0,+\infty): \exp_L:(0,t)\times\text{NS}(L)\to \mathcal{N}_{t}(L)\backslash L\ \text{is a homeomorphism.} \right\}
    \]
    be the \textit{injectivity radius} of $L$. It has a property related to the types of foils.
    \begin{lem}[\cite{li2024}]\label{lemma: property of inj radius in special cases}
        Let $\F$ be an embedded transnormal system of codimension one on a complete Riemannian manifold $M$, and suppose $L\in\F$ is a foil with injectivity radius $T<+\infty$. 
        \begin{enumerate}[(i)]
            \item If $L$ is an SR-foil or S-foil, then $\exp_L\big(T,\text{NS}(L)\big)$ is also an SR-foil or S-foil.
            \item If all foils in $\F$ are DR-foils, then $\exp_L(T,B(L))=\exp_L(-T,B(L))$. 
        \end{enumerate}
    \end{lem}
    \par    
    Moreover, defining the {\it diameter} of $\F$ as 
    \[
        D:=\sup\left\{d(L_1,L_2): L_1,L_2\in\F \right\},
    \]
    we have a criterion to classify the transnormal systems. 
    \begin{lem}[\cite{li2024}]\label{lemma: SR-foil and S-foil at most 2}
        Suppose $\F$ is an embedded transnormal system of codimension one on a complete connected Riemanniann manifold $M$. Let $N_{\text{C}}$ denote the sum of the numbers of SR-foils and S-foils contained in $\F$. Then, $N_{\text{C}}\le 2$, and more precisely, 
        \begin{itemize}
            \item if $D$ is infinity, then $N_\text{C}=1$ or $0$, and $M$ is diffeomorphic to the normal bundle of a foil $L\in \F$, where $L$ is the unique SR-foil or S-foil whenever $N_\text{C}=1$, or can be taken to be an arbitrary DR-foil whenever $N_\text{C}=0$;
            \item if $D$ is finite, then $N_\text{C}=2$ or $0$, and $M$ is diffeomorphic to the union of two normal disk bundles over foils $L,L'\in \F$
            glued together with their common boundary, where $L,L'$ are the only two foils whose normal sphere bundle is connected whenever $N_\text{C}=2$,
            or can be taken to be any pair of DR-foils satisfying $d(L,L')=D$ whenever $N_\text{C}=0$.
        \end{itemize}
    \end{lem}
    Based on this result, embedded transnormal systems of codimension one are classified into seven types as shown in Table \ref{table: seven types}, with a detailed discussion provided in \cite{li2024}. 
    \begin{table}[htbp]
        \centering
        \caption{Properties of the seven types.}
        \label{table: seven types}
        \renewcommand\arraystretch{1.2}
        \begin{tabular}{|c|c|c|c|c|}
        \noalign{\hrule height 1pt}
        Type & $N_{\text{SR}}$ & $N_{\text{S}}$ & Diameter of $\mathcal{F}$ & Foil space $M/\mathcal{F}$ \\ 
        \hline
        {\it Cylindrical} & 0               & 0              &  $+\infty$                &    $\mR$             \\ 
        \hline
        {\it Planar}      & 0               & 1              &  $+\infty$                &    $[0,+\infty)$     \\ 
        \hline
        {\it Twisted-cylindrical} & 1       & 0              &  $+\infty$                &    $[0,+\infty)$     \\ 
        \hline
        {\it Toric}       &  0              & 0              &  $T<+\infty$              &    $\mR/2T\mathbb{Z}$   \\ 
        \hline
        {\it Spherical}   &  0              & 2              &  $T<+\infty$              &    $[0,T]$           \\ 
        \hline
        {\it Real-projective}   &  1        & 1              &  $T<+\infty$              &    $[0,T]$           \\ 
        \hline
        {\it Klein-bottled}   &  2          & 0              &  $T<+\infty$              &    $[0,T]$           \\ 
        \noalign{\hrule height 1pt}
        \end{tabular}
    \end{table}  
    Among all types, if the foil space $M/\F$ has an endpoint, then the corresponding foil is either an SR-foil or an S-foil. 

\section{Covering and lifting on transnormal systems}\label{sec: covering and lifting}
    We explore the relationships among transnormal systems. Throughout this section, compactness of the ambient manifolds is not assumed.
    \subsection{Basic definitions and properties}
    We introduce some actions on the manifolds that are adapted to the transnormal systems. 
    \begin{defi}\label{def: foil-tran and foil-inv}
        Let $\F$ be an embedded transnormal system of codimension one on a complete Riemannian manifold $M$. 
        \begin{enumerate}[(i)]
            \item A diffeomorphism $\tau$ on $M$ is called a {\it foil-translation} with respect to $\F$, if there exists a foil $L\in\F$ and a positive real number $t_0$ such that 
            \begin{equation*}
                \tau\big(\exp_{L}(t,B(L))\big)=\exp_{L}\big(t+t_0,B(L)\big)
            \end{equation*}
            holds for any $t\in\mR$. 
            \item A diffeomorphism $\sigma$ on $M$ is called a {\it foil-reflection} with respect to $\F$, if there exists a DR-foil $L\in\F$ such that 
            \begin{equation}\label{equ: def of foil-reflection}
                \sigma\big(\exp_{L}(t,B(L))\big)=\exp_{L}(-t,B(L))
            \end{equation}
            holds for any $t\in\mR$. This DR-foil $L$ is called the {\it mirror foil} of the foil-reflection $\sigma$. 
        \end{enumerate} 
    \end{defi}
    \par 
    \begin{rmk}
        In the definition of a foil-reflection, it is essential that the mirror foil be a DR-foil. If $L$ is an SR-foil or an S-foil, then the identity map satisfies equation \myeqref{equ: def of foil-reflection}. However, the identity map clearly should not be regarded as a foil-reflection.
    \end{rmk}
    \par 
    We have known the foil space $M\big/\F$ is homeomorphic to one of $\mR$, $[0,+\infty)$, $[0,1]$ or $\mS^1$. From Definition \ref{def: foil-tran and foil-inv}, we can derive that the foil-translations and foil-reflections induce the corresponding translations and reflections on $M\big/\F$, which implies the following two facts:
    \begin{itemize}
        \item If $\tau$ is a foil-translation with respect to $\F$, then the foil space $M\big/\F$ must have no endpoints, or equivalently, $\F$ is of cylindrical or toric type.
        \item If $\sigma$ is a foil-reflection with respect to $\F$, then the foil space $M\big/\F$ must have either 0 or 2 endpoints, and the foils corresponding to the endpoints, if they exist, must be either both SR-foils or both S-foils. Thus, $\F$ is of cylindrical, toric, spherical or Klein-bottled type. 
    \end{itemize}
    \par 
    For an embedded transnormal system $\F$ of codimension one, suppose that $L_1$ and $L_2$ are distinct foils in $\F$, each being either an SR-foil or an S-foil. According to Lemma \ref{lemma: property of inj radius in special cases}, these two foils share a common finite injectivity radius $T$, and
    \begin{equation}\label{equ: when two Non-two-sided-foil}
        \mathcal{N}_{T}(L_1)\cap \mathcal{N}_{T}(L_2)=\mathcal{N}_{T}(L_1)\backslash L_1=M\backslash\left(L_1\cup L_2\right).
    \end{equation}
    The inverse $V\mapsto -V$ on $\nsb{L_1}$, denoted by $\tilde{\sigma}_{L_1}$ here, induces an involution on $\mathcal{N}_{T}(L_1)$ as  
    \begin{equation*}
        \sigma_{L_1}\big(\exp_{L_1}(t,V)\big):=\exp_{L_1}(t,\tilde{\sigma}_{L_1}(V)), \quad (t,V)\in[0,T)\times\nsb{L_1}.
    \end{equation*}
    Similarly, we can define $\tilde{\sigma}_{L_2}$ on $\nsb{L_2}$ and $\sigma_{L_2}$ on $\mathcal{N}_{T}(L_2)$. Given $t_0\in(0,T)$ and $V\in\nsb{L_1}$, it follows from \myeqref{equ: when two Non-two-sided-foil} that $\exp_{L_1}(t_0,V)\in \mathcal{N}_{T}(L_2)\backslash L_2$, and thus, there exists a unique $V\uppri\in \text{NS}(L_1)$ such that $\sigma_{L_2}\big(\exp_{L_1}(t_0,V)\big)=\exp_{L_1}(t_0,V\uppri)$. By a standard argument on normal geodesics, $V\uppri$ is independent of the choice of $t_0\in(0,T)$. In other words, the map $V\mapsto V\uppri$ defines an involution on $\text{NS}(L_1)$ induced by $\sigma_{L_2}$, which we also denote by $\tilde{\sigma}_{L_2}$, and it satisfies
    \begin{equation*}
        \exp_{L_1}(t,\tilde{\sigma}_{L_2}(V))=\sigma_{L_2}\big(\exp_{L_1}(t,V)\big),
    \end{equation*}
    for any $t\in(0,T)$ and $V\in\text{NS}(L_1)$. It is worth noting that there is no ambiguity between the involution $\tilde{\sigma}_{L_2}$ on $\text{NS}(L_1)$ and the natural involution $\tilde{\sigma}_{L_2}$ on $\text{NS}(L_2)$, since they act on distinct domains.
    \par 
    We now turn to the relationships among the cylindrical, toric, and Klein-bottled types.

    \subsection{Relationships among transnormal systems of various types}
    Now, we establish some relationships among transnormal systems of various types through foil-translations and foil-reflections. 
    \begin{prop}\label{prop: non-essential to essential}
        Suppose $\F$ is an embedded transnormal system of codimension one on a complete Riemannian manifold $M$. Then, 
        \begin{enumerate}[(i)]
            \item $\F$ is of toric type, if and only if there exists another complete Riemannian manifold $\tilde{M}$ equipped with a transnormal system $\tilde{\mathcal{F}}$ of cylindrical type, and an isometric foil-translation $\tau$ with respect to $\tilde{\mathcal{F}}$ such that $\tilde{M}\big/\langle\tau\rangle=M$.
            \item $\F$ is of twisted-cylindrical type, if and only if there exists another complete Riemannian manifold $\tilde{M}$ equipped with a transnormal system $\tilde{\mathcal{F}}$ of cylindrical type, along with an isometric foil-reflection $\sigma$ with respect to $\tilde{\mathcal{F}}$ that is a fixed-point-free involution on $\tilde{M}$, such that $\tilde{M}\big/\langle\sigma\rangle=M$.
            \item $\F$ is of Klein-bottled type, if and only if there exists another complete Riemannian manifold $\tilde{M}$ equipped with a transnormal system $\tilde{\mathcal{F}}$ of cylindrical type, along with two isometric foil-reflections $\sigma_1$ and $\sigma_2$ with distinct mirror foils with respect to $\tilde{\mathcal{F}}$ those both are fixed-point-free involutions on $\tilde{M}$, such that $\tilde{M}\big/\langle\sigma_1,\sigma_2\rangle=M$.
            \item $\F$ is of real-projective type, if and only if there exists another complete Riemannian manifold $\tilde{M}$ equipped with a transnormal system $\tilde{\mathcal{F}}$ of spherical type, along with an isometric foil-reflection $\sigma$ with respect to $\tilde{\mathcal{F}}$ that is a fixed-point-free involution on $\tilde{M}$, such that $\tilde{M}\big/\langle\sigma\rangle=M$.
        \end{enumerate}
    \end{prop}
    \begin{proof}
        We provide a proof of (iii), as the proofs of (i) and (ii) are similar and relatively simpler. 
        \par   
        Suppose $(M,\metr)$ admits a transnormal system $\{F_t\}_{t\in[0,\frac{t_0}{2}]}$ of Klein-bottled type induced by a transnormal function $f$. Denote $L_1=F_{0}$ and $L_2=F_{\frac{t_0}{2}}$ the SR-foils. We equip $\tilde{M}=\mR\times B(L_1)$ with a Riemannian metric to be determined. It follows that $\exp_{L_1}:\tilde{M}\to M$ is a covering map, and the function $f\circ \exp_{L_1}$ is a transnormal function on $(\tilde{M},\exp_{L_1}\upst\metr)$ inducing the transnormal system $\{\tilde{F}_t=\{t\}\times \text{NS}(L_1)\}_{t\in\mR}$. Define $\sigma_1$ and $\sigma_2$ on $\tilde{M}$ as $\sigma_1(t,V)=(-t,\tilde{\sigma}_{L_1}(V))$ and $\sigma_2(t,V)=(t_0-t,\tilde{\sigma}_{L_2}(V))$. Then, $\sigma_1$ and $\sigma_2$ are both isometric foil-reflections with distinct mirror foils with respect to $\tilde{\F}$ and fixed-point-free involutions on $\tilde{M}$. Obviously, $\exp_{L_1}=\exp_{L_1}\circ\sigma_1$. For $t\in(0,\frac{t_0}{2})$, we have 
        \begin{equation}\label{equ: normal geodesic is unique through SR-foil}
            \exp_{L_1}\circ\sigma_2(t,V)=\exp_{L_1}(t_0-t,\tilde{\sigma}_{L_2}(V))=\sigma_{L_2}\big(\exp_{L_1}(t_0-t,V)\big)=\exp_{L_1}(t,V),
        \end{equation}
        where the last equality follows from the uniqueness of the normal geodesic through $\exp_{L_1}(\frac{t_0}{2},V)$. A straightforward discussion of normal geodesics suggests that \myeqref{equ: normal geodesic is unique through SR-foil} holds for any $(t,V)\in\tilde{M}$. Taking the injectivity radius $T=\frac{t_0}{2}$ of $L_1$ and $L_2$ into consideration, we obtain $\tilde{M}\big/\langle\sigma_1,\sigma_2\rangle=M$. 
        \par    
        Conversely, assume $\tilde{\mathcal{F}}=\{\tilde{F}_t\}_{t\in\mR}$ is a transnormal system of cylindrical type, and $\sigma_1(\tilde{F}_t)=\tilde{F}_{-t}$, $\sigma_2(\tilde{F}_t)=\tilde{F}_{t_0-t}$ for $t_0>0$. Then, $G:=\langle\sigma_1,\sigma_2\rangle$ is a discrete group acting freely and isometrically on $\tilde{M}$. It follows that the quotient manifold $M=\tilde{M}\big/\langle\tau\rangle$ inherits a Riemannian metric $\metr$ that is locally isometric to $\tilde{\metr}$ via the natural projection $\pi:\tilde{M}\to M$. We obtain 
        \[
            \pi(\tilde{F}_t)=\pi\left(\bigcup_{k\in\mZ}\left(\tilde{F}_{kt_0+t}\cup\tilde{F}_{kt_0-t}\right)\right). 
        \]
        Thus, for any $t\in[0,\frac{t_0}{2}]$, we donote $F_{t}=\pi(\tilde{F}_t)$, and the collection $\mathcal{F}=\{F_{t}\}_{t\in[0,\frac{t_0}{2}]}$ forms a transnormal system on $M$ for the reason that the geodesics in $M$ locally retain the same properties as those in $\tilde{M}$. For any $t\in(0,\frac{t_0}{2})$ and sufficiently small $s>0$, the restriction of $\pi$ to $s$-neighborhood $\mathcal{N}_s(\tilde{F}_{t})$ of $\tilde{F}_{t}$ is an isometry, implying that all the foils except $F_0$ and $F_{\frac{t_0}{2}}$ in $\mathcal{F}$ are embedded DR-foils. The foil $F_0=\pi(\tilde{F}_0)$ has the same dimension as $\tilde{F}_0$, and its $s$-tube $\pi(\tilde{F}_{s}\cup\tilde{F}_{-s})=F_s$ is connected, which implies that $F_0$ is an SR-foil, and so is $F_{\frac{t_0}{2}}$. Hence, $\F$ is of Klein-bottled type. 
        \par   
        For (iv), we first note that it requires a different two-sheeted covering of $M$, which is described as follows. Let $M_1$ be a subset of normal bundle $\text{N}(L_1)$ defined as 
        \[
            M_1:=\left\{V\in \text{N}(L_1):\ |V|\le T\right\},
        \]
        and $M_2$ is a copy of $M_1$. Then the interior of $M_1$ is diffeomorphic to $\mathcal{N}_T(L_1)\subset M$, while the boundary is diffeomorphic to $\{T\}\times \text{NS}(L_1)$. Glue $M_1$ and $M_2$ along their boundary by attaching map 
        \[
            \phi:\partial M_1\simeq\{T\}\times \text{NS}(L_1)\to\partial M_2\simeq\{T\}\times \text{NS}(L_1),\quad \phi\left(T,V\right)=\left(T,\tilde{\sigma}_{L_2}(V)\right). 
        \]
        Noting that \myeqref{equ: normal geodesic is unique through SR-foil} also holds for real-projective type and letting $\tilde{\phi}(t,V):=\left(t_0-t,\tilde{\sigma}_2(V)\right)$, we have $\exp_{L_1}(t,V)=\exp_{L_1}\circ\tilde{\phi}(t,V)$ for any $t\in\mR$ and $V\in\nsb{L_1}$, which implies that the obtained manifold, denoted by $\tilde{M}$, is smooth and is a two-sheeted covering of $M$. The rest of the proof proceeds in the same manner as in (iii).   
    \end{proof}
    \begin{rmk}
        For completeness, we elaborate on a step in the proof above that was stated without full detail, specifically the fact that in (iii) of Proposition \ref{prop: non-essential to essential}, the group $G$ generated by $\sigma_1$ and $\sigma_2$ acts freely on $\tilde{M}$.  Firstly, $G$ is rewritten as
        \begin{align*}
            G&=\langle\sigma_1,\sigma_2\rangle\\
            &=\left\{(\sigma_2\circ\sigma_1)^k,(\sigma_2\circ\sigma_1)^k\circ\sigma_2:k\in\mathbb{Z}\right\}.
        \end{align*}
        Note that  
        \[
            (\sigma_2\circ\sigma_1)^k(F_t)=F_{t+kt_0},
        \]
        which implies that $(\sigma_2\circ\sigma_1)^k$ has no fixed point unless $k=0$. As to $(\sigma_2\circ\sigma_1)^k\circ\sigma_2\in G$, it has no fixed point because we have 
        \begin{align*}
            &\left(\sigma_2\circ\sigma_1\right)^{2N}\circ\sigma_2\,(x)=x \\
            \Leftrightarrow\ & \sigma_2\left(\left(\sigma_1\circ\sigma_2\right)^{N}(x)\right)=\left(\sigma_1\circ\sigma_2\right)^{N}(x),
        \end{align*}
        and
        \begin{align*}
            &\left(\sigma_2\circ\sigma_1\right)^{2N+1}\circ\sigma_2\,(x)=x \\
            \Leftrightarrow\ & \sigma_1\left(\left(\sigma_2\circ\sigma_1\right)^{N}\circ\sigma_2\,(x)\right)=\left(\sigma_2\circ\sigma_1\right)^{N}\circ\sigma_2\,(x), 
        \end{align*}
        for any $N\in\mZ$. 
    \end{rmk}
    \par    
    \begin{coro}\label{coro: Klein-bottled to toric}
        A complete Riemannian manifold $M$ admits a transnormal system $\F$ of Klein-bottled type, if and only if there exists another complete Riemannian manifold $\bar{M}$ equipped with a transnormal system $\bar{\F}$ of toric type, along with an isometric foil-reflection $\bar{\sigma}$ with respect to $\bar{\F}$ that is a fixed-point-free involution on $\bar{M}$, such that $\bar{M}\big/\langle\bar{\sigma}\rangle=M$. 
    \end{coro}
    \begin{proof}
        This conclusion immediately follows from the fact that the composition $\tau:=\sigma_2\circ\sigma_1$ of $\sigma_1$ and $\sigma_2$ in (iii) of Proposition \ref{prop: non-essential to essential} is exactly an isometric foil-translation with respect to the transnormal system of cylindrical type. 
    \end{proof}
    \par
    Since a Riemannian manifold admitting a transnormal system of cylindrical type has a particularly simple topology, namely, it is diffeomorphic to $\mR\times F$ (see \cite{li2024} for details), Proposition \ref{prop: non-essential to essential} (i) illustrates that a Riemannian manifold $M$ admitting a transnormal system $\F$ of toric type is diffeomorphic to the mapping torus of a foil, i.e.,
    \begin{equation*} 
        M\simeq \frac{\mR\times F}{(t,x)\sim(t+1,\tilde{\tau}(x))}\simeq \frac{[0,1]\times F}{(0,x)\sim(1,\tilde{\tau}(x))}=: \mathcal{M}_{\tilde{\tau}},
    \end{equation*}
    where $F$ is diffeomorphic to any foil in $\F$ and $\tilde{\tau}$ is a diffeomorphism of $F$. Meanwhile, $M$ is also regarded as an $F$-bundle over the circle induced by the diffeomorphism $\tilde{\tau}$.
    \par 
    \begin{ex}\label{example: case 1.2}
        Suppose $M=\mR^2/\Gamma(c)$ is a flat torus, where $\Gamma(c)=\{k_1(1,c)+k_2(0,1)\in\mR^2:k_1,k_2\in\mathbb{Z}\}$ is a lattice with $c\in[0,1)$. Denote $\pi:\mR^2\to\mR^2/\Gamma(c)$ the natural projection. For any $c\in[0,1)$, the set $F_{[x]}=\pi(\{x\}\times\mR)$ is clearly an embedded geodesic circle in $M$. The collection $\mathcal{F}=\{F_{[x]}\}_{[x]\in\mR/\mathbb{Z}}$ then defines a transnormal system of toric type on $M$. If $\gamma_{y_0}$ denotes the $y_0$-line in $\mR^2$, then $\pi\circ\gamma_{y_0}$ is a normal geodesic for $(M,\mathcal{F})$. The associated $(\tilde{M},\tilde{\F})$ in Proposition \ref{prop: non-essential to essential} (i) is given by 
        \[
            \tilde{M}=\mR\times \left(\mR/\mathbb{Z}\right),\quad \tilde{\mathcal{F}}=\{\tilde{F}_t=\{t\}\times \mR/\mathbb{Z}\}_{t\in\mR},
        \]
        and the action $\tau$ is given by $\tau(x,[y])=(x+1,[y+c])$. 
        \begin{figure}[htbp]
            \centering
            \includegraphics[width=0.6\linewidth]{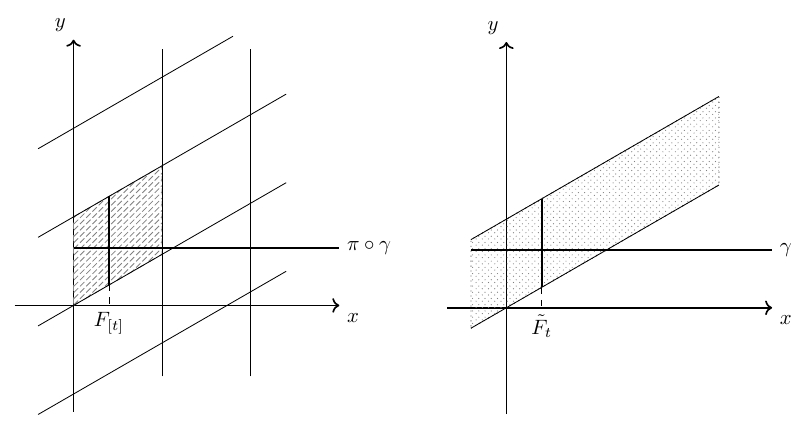}
            \caption{$\mathcal{F}$ of toric type and its covering.}
            \label{figure: torus and type R-C}
        \end{figure}
        The choice of $c$ influences the geometry of $M$. The normal geodesic $\pi\circ\gamma_{y_0}$ is dense in $M$ when $c$ is irrational, whereas it is closed when $c$ is rational.
    \end{ex}
    \par 
    Similarly, from Proposition \ref{prop: non-essential to essential} (iii), we obtain that a Riemannian manifold $M$ admitting a transnormal system $\F$ of Klein-bottled type is diffeomorphic to 
    \begin{equation}\label{equ: def of K_sigma12}
        \begin{aligned}
            \mathcal{K}_{\tilde{\sigma}_1;\tilde{\sigma}_2}:=&\frac{\mR\times F}{(t,x)\sim(-t,\tilde{\sigma}_1(x)),(t,x)\sim(1-t,\tilde{\sigma}_2(x))}\\
            \simeq &\frac{[0,\frac12]\times F}{(0,x)\sim(0,\tilde{\sigma}_1(x)),(\frac12,x)\sim(\frac12,\tilde{\sigma}_2(x))},
        \end{aligned}
    \end{equation}
    where $F$ is diffeomorphic to any DR-foil in $\F$ and $\tilde{\sigma}_1,\tilde{\sigma}_2$ are both fixed-point-free involution of $F$.

    \subsection{Lifting of transnormal systems along covering maps}
    Let $\F$ be a transnormal system on a Riemannian manifold $(M,\metr)$ induced by a transnormal function $f$, and let $\hat{\pi}:\hat{M}\to M$ be a regular Riemannian covering map with covering transformation group $G=G(\hat{M},\hat{\pi})$. The function $f\circ\hat{\pi}$ induces a transnormal system $\hat{\F}$ on $\hat{M}$. For any $g\in G$ and $\hat{F}\in\hat{\mathcal{F}}$, $g(\hat{F})$ is another connected component of $f\circ\hat{\pi}$, thus also a foil in $\hat{\mathcal{F}}$. In other words, each transformation in $G$ is a foil-preserving isometry of $(\hat{M},\hat{\mathcal{F}})$. Consider the subset 
    \begin{equation}\label{equ: normal subgroup fixing foils}
        G_0=G_0(\hat{M},\hat{\pi}):=\left\{g\in G: g(\hat{F})=\hat{F},\forall\,\hat{F}\in\hat{\mathcal{F}}\right\},
    \end{equation}
    which is indeed a normal subgroup of $G$. Then, it follows that the projection $p:\hat{M}\to\tilde{M}=\hat{M}\big/G_0$ is a regular covering map, and $\tilde{\pi}:\tilde{M}\to M$ is also a regular covering map whose covering transformation group $G(\tilde{M},\tilde{\pi})$ is isomorphic to the quotient group $G\big/G_0$. Consequently, $f\circ \tilde{\pi}$ induces a transnormal system $\tilde{\mathcal{F}}$ on $\tilde{M}$ such that $p(\hat{F})\in\tilde{\mathcal{F}}$ for any $\hat{F}\in\hat{\mathcal{F}}$, and $\tilde{\pi}(\tilde{F})\in\mathcal{F}$ for any $\tilde{F}\in\tilde{\mathcal{F}}$. Precisely, we have 
    \begin{prop}\label{prop: decom of an arbitrary covering}
        A regular covering map $\hat{\pi}:(\hat{M},\hat{\mathcal{F}})\to (M,\mathcal{F})$ can be uniquely decomposed into a regular covering sequence 
        \begin{equation}\label{equ: decom of an arbitrary covering}
            (\hat{M},\hat{\mathcal{F}})\xrightarrow{p}(\tilde{M},\tilde{\mathcal{F}})\xrightarrow{\tilde{\pi}}(M,\mathcal{F})
        \end{equation}
        satisfying the following conditions:
        \begin{enumerate}[(C1)]
            \item for every $g\in G(\hat{M},p)$ and $\hat{F}\in\hat{\F}$, it holds that $g(\hat{F})=\hat{F}$;
            \item if $\phi\in G(\tilde{M},\tilde{\pi})$ satisfies $\phi(\tilde{F})=\tilde{F}$ for any $\tilde{F}\in\tilde{\F}$, then $\phi$ is the identity. 
        \end{enumerate}
    \end{prop}
    \begin{proof}
        The existence of the decomposition has already been explained in the previous discussion; now, it suffices to show the uniqueness. Assume $(\hat{M},\hat{\F})\xrightarrow{p\uppri}(\tilde{M}\uppri,\tilde{\F}\uppri)\xrightarrow{\tilde{\pi}\uppri}(M,\F)$ is another regular covering sequence that satisfies the conditions. By condition (C1), $G(\hat{M},p\uppri)$ is a subgroup of $G_0$ in \myeqref{equ: normal subgroup fixing foils}. We claim that $G(\hat{M},p\uppri)=G_0$. Otherwise, suppose $\phi\in G_0\big\backslash G(\hat{M},p\uppri)$. Since $\tilde{\pi}\uppri\circ p\uppri\circ \phi=\tilde{\pi}\uppri\circ p\uppri$, there exists some $\psi_x\in G(\tilde{M}\uppri,\tilde{\pi}\uppri)$ for any $x\in\hat{M}$ such that $\psi_x\circ p\uppri(x)=p\uppri\circ \phi(x)$. Due to the discontinuity of $G(\tilde{M}\uppri,\tilde{\pi}\uppri)$, we can derive that $\psi:=\psi_x$ is independent of $x\in\hat{M}$, and thus $\psi\circ p\uppri=p\uppri\circ \phi$. Because $\phi\notin G(\hat{M},p\uppri)$, the transformation $\psi$ is not identity. However, $\phi(\hat{F})=\hat{F}$ means $p\uppri(\hat{F})$ and $p\uppri\circ \phi(\hat{F})$ are the same in $\tilde{\F}\uppri$, and thus $\psi$ fixes every foil in $\tilde{\F}\uppri$, which is a contradiction to the condition (C2). This proves the claim. It immediately follows that $(\tilde{M}\uppri,\tilde{\F}\uppri)$ and $(\tilde{M},\tilde{\F})$ coincide, indicating that these two regular covering sequences are indeed the same. 
    \end{proof}
    \par 
    Moreover, the covering map $p$ in \myeqref{equ: decom of an arbitrary covering} preserves the types of foils. We immediately have
    \begin{coro}
        In sequence \myeqref{equ: decom of an arbitrary covering}, the transnormal systems $\hat{\F}$ and $\tilde{\F}$ are of the same type. 
    \end{coro}
    \par 
    The following results will show that, for certain types of transnormal systems, the decomposition in sequence \myeqref{equ: decom of an arbitrary covering} is partially degenerate.
    \begin{prop}\label{prop: lift of essential is essential}
        Suppose $\hat{\pi}:(\hat{M},\hat{\F})\to(M,\F)$ is a regular covering map, and consider the regular covering sequence \myeqref{equ: decom of an arbitrary covering} derived from $\hat{\pi}$. If $\F$ is of cylindrical, planar or spherical type, then the covering map $\tilde{\pi}$ is trivial; equivalently, we have $p=\hat{\pi}$. 
    \end{prop}
    \begin{proof}
        First of all, we show that $\hat{F}$ contains no SR-foil. Otherwise, suppose $\hat{L}\in\hat{\F}$ is an SR-foil. For some $\hat{V}\in\nsb{\hat{L}}$ and sufficiently small $t>0$, there exists a foil $\hat{F}_t\in\hat{\F}$ such that $\exp_{\hat{L}}(\pm t,\hat{V})\in \hat{F}_t$. Since $\hat{\pi}$ is locally isometric, $L=\hat{\pi}(\hat{L})$ is a hypersurface, and 
        \[
            \exp_L(\pm t,\hat{\pi}\dwst\hat{V})\in\hat{\pi}(\hat{F}_t)\in\F. 
        \]
        Therefore, $L$ is an SR-foil in $\F$, which is impossible. 
        \par 
        Denote $I=M\big/\F$ and $\hat{\Lambda}=\hat{M}\big/\hat{\F}$. According to the types of $\F$ and $\hat{\F}$, $I$ is one of $\mR$, $[0,+\infty)$ or $[0,T]$, while $\hat{\Lambda}$ is one of $\mR$, $[0,+\infty)$, $[0,\hat{T}]$ or $\mR/\hat{t}_0\mZ$, and each endpoint of $I$ and $\hat{\Lambda}$, if it exists, corresponds to an S-foil. Denote $\hat{\pi}_{\#}:\hat{\Lambda}\to I$ the induced map of $\hat{\pi}$. It has the following properties: 
        \begin{itemize}
            \item $\hat{\pi}_{\#}$ maps the interior of $\hat{\Lambda}$ to the interior of $I$, the endpoints to the endpoints;
            \item the restriction of $\hat{\pi}_{\#}$ to the interior of $\hat{\Lambda}$ is locally homeomorphic. 
        \end{itemize}
        These imply that $\hat{\pi}_{\#}$ is a homeomorphism between $\hat{\Lambda}$ and $I$, and thus $G_0(\hat{M},\hat{\pi})=G(\hat{M},\hat{\pi})$. This completes the proof. 
    \end{proof}
    \par 
    Proposition \ref{prop: lift of essential is essential} contrasts with Proposition \ref{prop: non-essential to essential}, as transnormal systems of cylindrical, planar, and spherical types can only be lifted to transnormal systems of the same types, whereas the remaining four types behave differently. This observation motivates the following definition:
    \begin{defi}
        A transnormal system is called {\it essential}, if it is of cylindrical, planar or spherical type. For a transnormal system $\F$ of toric, twisted cylindrical, Klein-bottled or real-projective type on $M$, the corresponding Riemannian manifold $\tilde{M}$ and essential transnormal system $\tilde{\F}$, as given in Proposition \ref{prop: non-essential to essential}, are called the {\it essential cover} with respect to $\F$ and the {\it essential lift} of $\F$, respectively. As an extension, the {\it essential cover} and {\it essential lift} of $(M,\F)$ when $\F$ is essential are defined to be $M$ and $\F$ themselves, respectively. 
    \end{defi}
    \par 
    Next, we discuss a special case of regular covering maps: the universal cover. For $(M,\F)$, we denote $M_{\text{uni}}$ the universal covering space of $M$ via $\pi_{\text{uni}}:M_{\text{uni}}\to M$, and $\F_{\text{uni}}$ the lifting transnormal system on $M_{\text{uni}}$. 
    \begin{prop}\label{prop: the mid in decomposition of universal is essential}
        The pair $(\tilde{M},\tilde{\F})$ in sequence \myeqref{equ: decom of an arbitrary covering}, derived from the universal covering map $(M_{\text{uni}},\F_{\text{uni}})\xrightarrow{\pi_{\text{uni}}}(M,\F)$, are exactly the essential cover with respect to $\F$ and essential lift of $\F$. 
    \end{prop}
    \begin{proof}
        Denote $(M_{\text{ess}},\F_{\text{ess}})$ the essential cover and essential lift of $(M,\F)$. Clearly, 
        \[
            (M_{\text{uni}},\F_{\text{uni}})\xrightarrow{p}(M_{\text{ess}},\F_{\text{ess}})\xrightarrow{\pi_{\text{ess}}}(M,\F)
        \]
        is a regular Riemannian covering sequence, and it suffices to verify that conditions (C1) and (C2) in Proposition \ref{prop: decom of an arbitrary covering} are satisfied, or equivalently, that $G_0(M_{\text{uni}},\pi_{\text{ess}}\circ p)=G(M_{\text{uni}},p)$. By Proposition \ref{prop: lift of essential is essential}, we have $G(M_{\text{uni}},p)=G_0(M_{\text{uni}},p)$, and this is a subgroup of $G_0(M_{\text{uni}},\pi_{\text{ess}}\circ p)$. Suppose, for the sake of contradiction, that there exists $\phi\in G_0(M_{\text{uni}},\pi_{\text{ess}}\circ p)\big\backslash G_0(M_{\text{uni}},p)$. Using an argument analogous to the proof of Proposition \ref{prop: decom of an arbitrary covering}, one can obtain a non-trivial $\psi\in G_0(M_{\text{ess}},\pi_{\text{ess}})$. However, this leads to a contradiction, as $G(M_{\text{ess}},\pi_{\text{ess}})$ is either trivial or generated by specific foil-translations or foil-reflections. Thus, the proof is complete. 
    \end{proof}
    \par 
    \begin{prop}\label{prop: universal property of proper covering}
        For sequence \myeqref{equ: decom of an arbitrary covering} derived from arbitrary $(\hat{M},\hat{\F})\to (M,\mathcal{F})$, there exists a Riemannian covering map $\pi\uppri:M_{\text{ess}}\to\tilde{M}$ satisfying $\pi_{\text{ess}}=\tilde{\pi}\circ\pi\uppri$, that is, the following diagram commutates. 
        \begin{equation*}
            \begin{tikzcd}
                & (M_{\operatorname{ess}},\F_{\operatorname{ess}}) \arrow{d}{\pi\uppri}\arrow{dr}{\pi_{\operatorname{ess}}} &\\
                (\hat{M},\hat{\F})\arrow{r}{p}  &(\tilde{M},\tilde{\F})\arrow{r}{\tilde{\pi}} &(M,\F)
            \end{tikzcd}
        \end{equation*}
    \end{prop}
    \begin{proof}
        Given a regular covering sequence $\hat{M}_1\xrightarrow{\iota}\hat{M}_2\xrightarrow{\hat{\pi}_2}M$ with $\hat{\pi}_1=\hat{\pi}_2\circ\iota$, consider the decompositions of $\hat{M}_1\xrightarrow{\hat{\pi}_1}M$ and $\hat{M}_2\xrightarrow{\hat{\pi}_2}M$. Through a direct analysis of the quotient spaces, we obtain a regular covering map $[\iota]:\tilde{M}_1\to\tilde{M}_2$ such that the following diagram commutes:
        \begin{equation}\label{equ: diagram of covering sequence}
            \begin{tikzcd}
                \hat{M}_1 \arrow{r}{p_1} \arrow{dd}{\iota} & \tilde{M}_1 \arrow{dd}{[\iota]}\arrow{drr}{\tilde{\pi}_1} &&\\
                &&& M.\\
                \hat{M}_2\arrow{r}{p_2}  & \tilde{M_2}\arrow{urr}{\tilde{\pi}_2}&&
            \end{tikzcd}
        \end{equation}
        Substituting $\hat{M}_1$ for the universal cover $M_{\text{uni}}$ and $\hat{M}_2$ for the considered space $\hat{M}$, we complete the proof by Proposition \ref{prop: the mid in decomposition of universal is essential}.
    \end{proof}
    \par 
    In summary, the regular covering sequence \myeqref{equ: decom of an arbitrary covering} provides a unique decomposition of the covering map $\hat{\pi}$ into two covering maps, $p$ and $\tilde{\pi}$. The map $\tilde{\pi}$ corresponds to the lift of the foil space $M\big/\F$, while $p$ acts on each foil individually. The essential lift is characterized as the maximal lift of the foil space.

\section{Transnormal systems on three-dimensional manifolds}
    We focus on the classification of transnormal systems on compact three-manifolds in the equivariant sense of Definition \ref{def: equivalence of tran sys}. The geometry and topology of three-manifolds have been extensively studied, with notable contributions such as Thurston's geometrization. Throughout this section, we assume that $M$ is a compact three-manifold (without boundary) and $\F$ is an embedded transnormal system of codimension one on $M$. Since foil space $M/\F$ is compact, it follows that $\F$ must be of toric, Klein-bottled, spherical, or real-projective type. Each of these types will be discussed in detail.
    \subsection{Transnormal systems of toric type}\label{subsec: 3d and toric type}
    We firstly provide the criterion for equivalence of two transnormal systems of toric type. 
    \begin{prop}\label{prop: equiv of toric type}
        Let $\F$ and $\F\uppri$ be transnormal systems of toric type on Riemannian manifolds $M$ and $M\uppri$, respectively. They are equivalent if and only if there exists $F_{[t]}\in \F$, $F\uppri_{[t\uppri]}\in\F\uppri$, and two homeomorphisms $\varphi_1,\varphi_2:F_{[t]}\to F\uppri_{[t\uppri]}$ smoothly isotopic to each other, such that $\tilde{\tau}\uppri=\varphi_1\circ\tilde{\tau}\circ\varphi_2^{-1}$ or $\tilde{\tau}\uppri=\varphi_1\circ\tilde{\tau}^{-1}\circ\varphi_2^{-1}$ holds, where $\tilde{\tau}(x)$ is defined as the point where the directed normal geodesic, starting from $x\in F_{[t]}$, first return to $F_{[t]}$, and $\tilde{\tau}\uppri$ is defined similarly on $F\uppri_{[t\uppri]}$.   
    \end{prop}
    \begin{proof}
        Assume a diffeomorphism $\varphi:M\to M\uppri$ induces the equivalence between $\F=\{F_{[t]}\}_{[t]\in\mR/t_0\mZ}$ and $\F\uppri=\{F\uppri_{[t]}\}_{[t]\in\mR/t\uppri_0\mZ}$. By Proposition \ref{prop: non-essential to essential} (i), the essential cover of $M$ with respect to $\F$ is diffeomorphic to $\mR\times F$, and the covering transformation group is generated by a foil-translation $\tau$ with $\tau(s,x)=(s+t_0,\tau_F(x))$. Similarly, for $(M\uppri,\F\uppri)$, we have the essential cover diffeomorphic to $\mR\times F\uppri$ and generator $\tau\uppri$ of the covering transformation group with $\tau\uppri(s,x\uppri)=(s+t\uppri_0,\tau_{F\uppri}(x\uppri))$. Due to Proposition \ref{prop: universal property of proper covering}, $\varphi$ induces a diffeomorphism $\Phi:\mR\times F\to \mR\times F\uppri$ with $\Phi(s,x)=(h(s),\Phi_s(x))$ such that the following diagram commutates:
        \begin{equation*}
            \begin{tikzcd}
                \mR\times F \arrow{r}{\Phi} \arrow{d}{\tau} & \mR\times F\uppri \arrow{d}{\tau\uppri}\\
                \mR\times F\arrow{r}{\Phi}  & \mR\times F\uppri. 
            \end{tikzcd}
        \end{equation*}
        Without loss of generality, we can assume $h\uppri(s)\ge 0$, and then, 
        \begin{equation}\label{equ: condition of ess induced to below in toric type}
            \left\{
                \begin{aligned}
                    &h(s)+t_0\uppri=h(s+t_0),\\
                    &\tau_{F\uppri}\uppri=\Phi_{s+t_0}\circ\tau_{F}\circ\Phi_s^{-1}. 
                \end{aligned}
            \right.
        \end{equation}
        \par 
        Identifying $F$ with $F_{[t]}$, the transformation $\tilde{\tau}$ coincides with $\tau_F$. As $\Phi$ is smooth, $\varphi_1=\Phi_{s+t_0}$ is evidently smoothly isotopic to $\varphi_2=\Phi_{s}$.       
        \par 
        Conversely, if $\varphi_1$ is smoothly isotopic to $\varphi_2$, there exists a smooth isotopy $\Psi:\mR\times F\to F\uppri$ such that $\Psi(0,\cdot)=\varphi_1$ and $\Psi(1,\cdot)=\varphi_2$. Define $\Phi:\mR\times F\to \mR\times F\uppri$ by 
        \[
            \Phi(s,x)=\left(\frac{t\uppri_0}{t_0}\cdot s,\Phi_s(x)\right),
        \] 
        where $\Phi_s(x)$ is given by 
        \begin{equation*}
            \Phi_s(x)=\left\{ 
                \begin{aligned}
                    &\Psi(\chi(s),x),\quad &s\in[0,t_0],\\
                    &{\tau_{F\uppri}\uppri}^{k}\circ\Phi_{s-kt_0}\circ\tau_{F}^{-k}(x),\quad &s\in[kt_0,(k+1)t_0],
                \end{aligned}    
            \right.
        \end{equation*}
        and $\chi$ is a smooth function such that $\chi(s)=0$ for $s<\frac{t_0}{3}$ and $\chi(s)=1$ for $s>\frac{2t_0}{3}$. The map $\Phi$ establishes the equivalence between the essential lifts of $\F$ and $\F\uppri$, satisfying the equations in \myeqref{equ: condition of ess induced to below in toric type}. Consequently, $\F$ and $\F\uppri$ are equivalent. 
    \end{proof}
    \par 
    When $M$ is a compact 3-manifold and $\F$ is of toric type, the DR-foils in $\F$ are diffeomorphic to a closed surface. Combining Proposition \ref{prop: equiv of toric type} with the properties of transnormal systems of toric type in Section \ref{sec: covering and lifting}, we immediately have 
    \begin{thm}\label{thm: equiv of tran sys of toric type}
        Let $S$ be a closed surface, and $\mathcal{E}_{\text{toric}}(S)$ collect all the equivalence classes of transnormal systems of toric type on compact 3-manifolds with DR-foils homeomorphic to $S$. Then, $\mathcal{E}_{\text{toric}}(S)$ is in one-to-one correspondence with the set
        \begin{equation*}
            \eMCG{S}\big/\sim,
        \end{equation*} 
        where $\eMCG{S}$ is the extended mapping class group of $S$ and $\sim$ is an equivalence relation generated by inverses and conjugates. 
    \end{thm}
    \begin{rmk}
        A detailed discussion of the extended mapping class group can be found in \cite[Chap.8]{farb2011primer}. The quotient set $\eMCG{S}\big/\sim$ does not necessarily form a group. In this paper, the equivalence class of the homeomorphism $\tilde{\tau}$ in $\eMCG{S}\big/\sim$ is denoted by $\llbracket\tilde{\tau}\rrbracket$. 
    \end{rmk}
    \par 
    \begin{ex}\label{example: periodic homeo inducing tran sys}
        Let $S$ be a closed surface, and $\tilde{\tau}$ be a homeomorphism of $S$ satisfying $\tilde{\tau}^{k}=id$ for some positive integer $k$. Given any Riemannian metric $\metr_{0}$ on $S$, the metric 
        \[
            \metr_S=\sum_{i=0}^{k-1}(\tilde{\tau}\upst)^i\,\metr_{0}
        \]
        is clearly invariant under the action of $\tilde{\tau}$. Consequently, the map $\tau(t,x)=(t+1,\tilde{\tau}(x))$ is an isometry on the product space $\left(\mR\times S,\dd{t^2}+\metr_S\right)$. By Proposition \ref{prop: non-essential to essential}, this induces a transnormal system $\F$ of toric type on $\mR\times S\big/\langle\tau\rangle$, whose equivalence class corresponds to $\llbracket\tilde{\tau}\rrbracket\in\eMCG{S}\big/\sim$. Moreover, all foils in $\F$ are totally geodesic. 
    \end{ex}
    \par 
    \begin{ex}
        When DR-foils are homeomorphic to the 2-sphere, the set $\eMCG{\mS^2}\big/\sim$ consists of only two elements, which corresponds to the identify $id_{\mS^2}$ and the antipodal map $\sigma_{\mS^2}$, respectively. The mapping torus $\mathcal{M}_{id_{\mS^2}}$ and $\mathcal{M}_{\sigma_{\mS^2}}$ are homeomorphic to $\mS^1\times\mS^2$ and the twisted $\mS^2$-bundle over $\mS^1$, respectively.
    \end{ex}
    \begin{ex}
        When DR-foils are homeomorphic to $\mR P^2$, the set $\eMCG{\mR P^2}\big/\sim$ consists of a single element $\llbracket id_{\mR P^2}\rrbracket$. The mapping torus $\mathcal{M}_{id_{\mR P^2}}$ is homeomorphic to $\mS^1\times\mR P^2$.
    \end{ex}
    \par 
    \begin{ex}\label{example: dehn twist of torus inducing tran sys}
        Assume $\tilde{\tau}$ is the Dehn twist of flat torus $T^2=\mR^2/\mathbb{Z}^2$. By definition (see \cite[Chap.3]{farb2011primer} for example), Dehn twist $\tilde{\tau}$ is isotopic to the linear transformation induced by the matrix
        \[
            D=\left(
                \begin{array}{cc}
                    1&1\\
                    0&1
                \end{array}
            \right)\in\operatorname{SL}(2,\mZ).
        \]
        Consider on $\mR\times T^2$ the Riemannian metric 
        \begin{equation}\label{equ: metric of nil geometry}
            \metr=\dd{t^2}+\dd{x^2}+\left(\dd{y}-t\,\dd{x}\right)^2,
        \end{equation}
        and then $\tilde{\F}=\left\{\{t\}\times T^2\right\}_{t\in\mR}$ is a transnormal system on $(\mR\times T^2,\metr)$. Define $\tau:\mR\times T^2\to \mR\times T^2$ by
        \[
            \tau(t,x,y)=\big(t+1,(x,y)\cdot D\big)=(t+1,x,x+y),
        \]
        which is indeed an isometric foil-translation with respect to $\tilde{\F}$.  
        Moreover, each hypersurface $\{t\}\times T^2$ has constant principal curvatures $\frac{1}{2}$ and $-\frac{1}{2}$. Thus, by Proposition \ref{prop: non-essential to essential}, there exists a CPC transnormal system $\F$ of toric type on $\mR\times T^2\big/\langle\tau\rangle$, whose equivalence class corresponds to $\llbracket\tilde{\tau}\rrbracket\in\eMCG{T^2}\big/\sim$. Note that the metric $\metr$ in \myeqref{equ: metric of nil geometry} is locally isometric to the Nil geometry. 
    \end{ex}
    \begin{ex}\label{example: Anosov homeo of torus inducing tran sys}
        Assume $\tilde{\tau}$ is an Anosov homeomorphism (see \cite[Chap.13]{farb2011primer} for example) of flat torus $T^2=\mR^2/\mathbb{Z}^2$. Similarly to Example \ref{example: dehn twist of torus inducing tran sys}, $\tilde{\tau}$ is isotopic to the linear transformation induced by a matrix $A\in\operatorname{SL}(2,\mZ)$ with two distinct positive eigenvalues, denoted by $\lambda(>1)$ and $\lambda^{-1}$. Equip $\mR\times T^2$ with metric 
        \begin{equation}\label{equ: metric of Anosov of torus}
            \metr=\ln^2\lambda\,\dd{t^2}+\lambda^{-2t}\dd{\xi_1^2}+\lambda^{2t}\dd{\xi_2^2},
        \end{equation}
        where $(\xi_1,\xi_2)$ are affine coordinates on $\mR^2$ with 
        \begin{equation}\label{equ: eigenvector of anosov}
            A\left(\pp{\xi_1}\right)=\lambda \pp{\xi_1},\quad A\left(\pp{\xi_2}\right)=\lambda^{-1} \pp{\xi_2}. 
        \end{equation}
        Then, the action 
        \[
            \tau(t,x,y)=\left(t+1,(x,y)\cdot A\right)
        \]
        is an isometric foil-translation with respect to $\left\{\{t\}\times T^2\right\}_{t\in\mR}$. We obtain a transnormal system $\F$ of toric type on $\mR\times T^2\big/\langle\tau\rangle$ by Proposition \ref{prop: non-essential to essential}, whose equivalence class corresponds to $\llbracket\tilde{\tau}\rrbracket\in\eMCG{T^2}\big/\sim$. Moreover, $\metr$ in \myeqref{equ: metric of Anosov of torus} is locally isometric to the Sol geometry, and all the foils have constant principal curvatures $1$ and $-1$. 
    \end{ex}
    \par 
    Next, we will focus on CPC transnormal system $\F$ of toric type on orientable compact 3-manifold $M$. The cases of non-orientable manifolds can be discussed in a similar manner. 
    Assume the foils in $\F$ are all homeomorphic to $S$. Since $M$ is covered by $\mR\times S$, the orientability of $M$ ensures that $S$ is orientable, and the associated foil-translation preserves orientation. Similarly, the set of such equivlence classes of transnormal systems is in one-to-one correspondence with 
    \begin{equation*}
        \MCG{S}\big/\sim, 
    \end{equation*}
    where $\MCG{S}$ is the mapping class group of the orientable closed surface $S$ and $\sim$ is the same equivalence relation as in Theorem \ref{thm: equiv of tran sys of toric type}. 
    \begin{prop}\label{prop: CPC for toric type}
        Let $M$ be an orientable compact 3-manifold, and $\F$ be a transnormal system of toric type on $M$ whose foils are homeomorphic to the orientable closed surface $S_g$ of genus $g$. 
        \begin{enumerate}[(i)]
            \item If $g\le 1$, then, up to equivlence, $\F$ is a CPC transnormal system, and meanwhile $M$ is locally isometric to $\mS^2\times\mR$, $\mR^3$, Nil geometry or Sol geometry.
            \item If $g\ge 2$, $\F$ is equivalently a CPC transnormal system if and only if $\tilde{\tau}$ is an orientation-preserving periodic homeomorphism of $S_g$ and the equivalence class of $\F$ corresponds to $\llbracket\tilde{\tau}\rrbracket\in\MCG{S_g}\big/\sim$. In this case, up to equivlence of CPC transnormal systems, $M$ is locally isometric to $\mH^2\times\mR$.   
        \end{enumerate}
    \end{prop}
    \begin{proof}
        Assume that $\F$ corresponds to $\llbracket\tilde{\tau}\rrbracket\in\MCG{S}\big/\sim$.
        \par 
        If $g=0$, the mapping class group of sphere $\mS^2$ is trivial, and so is $\MCG{\mS^2}\big/\sim$. $\F$ is equivalent to the transnormal system $\big\{\{\theta\}\times\mS^2\big\}_{\theta\in\mS^1}$ on $\mS^1\times\mS^2$. In this case, the conclusion naturally follows. 
        \par 
        If $g=1$, according to the classification of homeomorphisms on torus $T^2$ (see \cite[Chap.13]{farb2011primer}), $\tilde{\tau}$ is isotopic to a periodic, reducible, or Anosov homeomorphism. Within each periodic isotopy class, there exists a homeomorphism acting isometrically on a flat torus. Thus, by Example \ref{example: periodic homeo inducing tran sys}, it follows that, up to equivalence, $M$ is locally isometric to Euclidean $\mR^3$, and each foil is totally geodesic. 
        If $\tilde{\tau}$ is isotopic to a reducible homeomorphism, it can be assumed to be a multiple of a Dehn twist. By analogy with Example \ref{example: dehn twist of torus inducing tran sys}, this gives rise to a CPC transnormal system, with the ambient manifold being locally isometric to Nil geometry. If $\tilde{\tau}$ is isotopic to an Anosov homeomorphism, it corresponds to a CPC transnormal system, where the ambient manifold is locally isometric to Sol geometry, as in Example \ref{example: Anosov homeo of torus inducing tran sys}. 
        \par 
        When $g\ge 2$, we firstly assume $\tilde{\tau}$ is a periodic homeomorphism of $S_g$ and the equivalence class of $\F$ corresponds to $\llbracket\tilde{\tau}\rrbracket\in\MCG{S_g}\big/\sim$. By the uniformization theorem and the fact that conformal maps on the hyperbolic plane are isometries, $\tilde{\tau}$ is an isometry of $S_g$ with respect to a hyperbolic metric. It then follows from Example \ref{example: periodic homeo inducing tran sys} that, up to equivalence, the foils in $\F$ are totally geodesic and the ambient manifold is locally isometric to $\mH^2\times\mR$. 
        As for the converse, if $\F$ is a CPC transnormal system, we consider the essential cover $\big(\mR\times S_{g},\dd{t^2}+\metr_{S_{g}}(t)\big)$ with respect to $\F$, where each $\{t\}\times S_g$ has constant principal curvatures. By the line bundle version of the Poincaré-Hopf theorem (\cite[2.2. Theorem II]{hopf2003differential}), the two constant principal curvatures on $\{t\}\times S_g$ are equal, denoted by $\kappa(t)$. It follows that $X=\exp\{-\int_0^t\kappa(s)\dd{s}\}\pp{t}$ is a closed conformal vector field. According to \cite{montiel1999unicity}, the metric $\dd{t^2}+\metr_{S_{g}}(t)$ is factually a warped product metric, i.e., $\metr_{S_{g}}(t)=f^2(t)\metr_{S_{g}}$ for a fixed metric $\metr_{S_{g}}$ and a smooth positive function $f(t)$. Since the associated foil-translation $\tau(t,x)=(t+t_0,\tilde{\tau}(x))$ on $\mR\times S_{g}$ is an isometry with respect to $\dd{t^2}+f^2(t)\metr_{S_{g}}$, we obtain that $f(t)=f(t+t_0)$ and $\tilde{\tau}$ is isometric on $(S_g,\metr_{S_g})$. By the uniformization theorem, $\langle\tilde{\tau}\rangle$ is a conformal subgroup of $S_g$ with respect to some hyperbolic metric. Note that the conformal subgroups of any closed hyperbolic surface are finite, $\tilde{\tau}$ must be a periodic homeomorphism. 
    \end{proof}
    \subsection{Transnormal systems of Klein-bottled type}\label{subsec: 3d and Klein-bottled type}
    The criterion for equivalence of two transnormal systems of Klein-bottled type is as follows. 
    \begin{prop}\label{prop: equiv of Klein-bottled type}
        Let $\F$ and $\F\uppri$ be transnormal systems of Klein-bottled type on Riemannian manifolds $M$ and $M\uppri$, respectively. They are equivalent if and only if there exists a DR-foil $F_{t}\in \F$, a DR-foil $F\uppri_{t\uppri}\in\F\uppri$, and two homeomorphisms $\varphi_1,\varphi_2:F_{t}\to F\uppri_{t\uppri}$ smoothly isotopic to each other, such that 
        \begin{equation}\label{equ: equivalence relation between Inv times Inv}
            \left\{
                \begin{aligned}
                    &\tilde{\sigma}\uppri_1=\varphi_1\circ\tilde{\sigma}_1\circ\varphi_1^{-1},\\
                    &\tilde{\sigma}\uppri_2=\varphi_2\circ\tilde{\sigma}_2\circ\varphi_2^{-1},
                \end{aligned}
            \right.
            \quad \text{or}\quad 
            \left\{
                \begin{aligned}
                    &\tilde{\sigma}\uppri_1=\varphi_1\circ\tilde{\sigma}_2\circ\varphi_1^{-1},\\
                    &\tilde{\sigma}\uppri_2=\varphi_2\circ\tilde{\sigma}_1\circ\varphi_2^{-1},
                \end{aligned}
            \right.
        \end{equation}
        holds. Here, $\tilde{\sigma}_1(x)$ and $\tilde{\sigma}_2(x)$ are the points at which the normal geodesic starting from $x\in F_{t}$ in opposite directions first returns to $F_{t}$, whereas $\tilde{\sigma}\uppri_1$ and $\tilde{\sigma}\uppri_2$ are defined similarly on $F\uppri_{t\uppri}$. 
    \end{prop}
    \begin{rmk}
        According to the properties of transnormal systems of Klein-bottled type in Section \ref{sec: covering and lifting}, if $\nu$ is a unit vector field normal to a DR-foil $F_t$ and the normal geodesic $\gamma(t)$ with $\gamma(0)=x\in F_t, \dot{\gamma}(0)=\nu(x)$ first returns to $F_t$ at $\gamma(t_1)=\tilde{\sigma}_1(x)\in F_t$, then $\gamma(t)$ passes through an SR-foil at $\gamma(\frac{t_1}{2})$, and the tangent vector $\dot{\gamma}(t_1)$ is $-\nu\big(\gamma(t_1)\big)$. Therefore, $\tilde{\sigma}_1$ is an involution. Furthermore, the normal geodesic starting with tangent vector $-\nu$ will pass through the other SR-foil before returning to $F_t$.
    \end{rmk}
    \begin{proof}
        Assume a diffeomorphism $\varphi:M\to M\uppri$ induces the equivalence between $\F=\{F_{t}\}_{t\in[0,T]}$ and $\F\uppri=\{F\uppri_{t}\}_{t\in[0,T\uppri]}$. By Propositions \ref{prop: non-essential to essential} and \ref{prop: universal property of proper covering}, and a process analogous to Proposition \ref{prop: equiv of toric type}, we obtain a diffeomorphism $\Phi:(s,x)\mapsto (h(s),\Phi_s(x))$ between the essential covers $\mR\times F$ and $\mR\times F\uppri$ induced by $\varphi$. Let 
        \begin{equation*}
            \left\{
                \begin{aligned}
                    &\sigma_1(s,x)=(-s,\sigma_{\scriptscriptstyle 1,F}(x)),\\
                    &\sigma_2(s,x)=(2T-s,\sigma_{\scriptscriptstyle 2,F}(x)),
                \end{aligned}
            \right.
        \end{equation*}
        be the foil-reflections generating the covering transformation group of the essential cover of $(M,\F)$, and 
        \begin{equation*}
            \left\{
                \begin{aligned}
                    &\sigma\uppri_1(s,x\uppri)=(-s,\sigma\uppri_{\scriptscriptstyle 1,F}(x\uppri)),\\
                    &\sigma\uppri_2(s,x\uppri)=(2T\uppri-s,\sigma\uppri_{\scriptscriptstyle 2,F}(x\uppri)),
                \end{aligned}
            \right.
        \end{equation*}
        similarly for $(M\uppri,\F\uppri)$. Without loss of generality, we assume $h(0)=0$, $h\uppri(s)\ge 0$, and then 
        \begin{equation}\label{equ: condition of ess induced to below in Klein}
            \left\{
                \begin{aligned}
                    & h(s)+h(-s)=0,\quad h(T+s)+h(T-s)=2T\uppri,\\
                    & \sigma_{\scriptscriptstyle 1,F\uppri}\uppri=\Phi_{-s}\circ\sigma_{\scriptscriptstyle 1,F}\circ\Phi_{s}^{-1},\quad \sigma_{\scriptscriptstyle 2,F\uppri}\uppri=\Phi_{2T-s}\circ\sigma_{\scriptscriptstyle 2,F}\circ\Phi_{s}^{-1}.
                \end{aligned}
            \right.
        \end{equation}
        \par 
        Identifying $F$ with $F_{t}$, the transformations $\tilde{\sigma}_1$ and $\tilde{\sigma}_2$ coincide with $\sigma_{\scriptscriptstyle 1,F}$ and $\sigma_{\scriptscriptstyle 2,F}$, respectively. The conclusion follows with $\varphi_1=\Phi_0$ and $\varphi_2=\Phi_{T}$.
        \par 
        The proof of the converse part is nearly identical to that of Proposition \ref{prop: equiv of toric type}, except that the function $\Phi_s$ is now defined as 
        \begin{equation*}
            \Phi_s(x)=\left\{ 
                \begin{aligned}
                    &\Psi(\chi(s),x),\quad &s\in[0,T],\\
                    &\sigma_{\scriptscriptstyle 2,F\uppri}\uppri\circ\Phi_{2T-s}\circ\sigma_{\scriptscriptstyle 2,F}(x),\quad &s\in[T,2T],\\
                    &\big(\sigma_{\scriptscriptstyle 2,F\uppri}\uppri\circ\sigma_{\scriptscriptstyle 1,F\uppri}\uppri\big)^{k}\circ\Phi_{s-2kT}\circ\big(\sigma_{\scriptscriptstyle 1,F}(x)\circ\sigma_{\scriptscriptstyle 2,F}\big)^{k}(x),\quad &s\in[2kT,2(k+1)T],
                \end{aligned}    
            \right.
        \end{equation*}
        which satisfies the equations in \myeqref{equ: condition of ess induced to below in Klein}. 
    \end{proof}
    \par 
    For closed surface $S$, denote by $\operatorname{Inv}_0(S)$ the set of all the fixed-point-free involutions of $S$, that is, 
    \[
        \inv{S}:=\{\sigma\in \operatorname{Homeo}(S):\,\sigma^2=id\ \text{and $\sigma$ has no fixed point}\}.
    \]
    Note that $\tilde{\sigma}_1$ and $\tilde{\sigma}_2$ are both in $\inv{F_t}$. Combining Proposition \ref{prop: equiv of Klein-bottled type} and the properties of transnormal systems of Klein-bottled type, we obtain
    \begin{thm}\label{thm: equiv of tran sys of klein-bottled type}
        Let $\mathcal{E}_{\text{kl}}(S)$ collect all the equivalence classes of transnormal systems of Klein-bottled type on compact 3-manifolds with DR-foils homeomorphic to $S$. Then $\mathcal{E}_{\text{kl}}(S)$ is in one-to-one correspondence with the set
        \begin{equation*}
            \inv{S}\times\inv{S}\big/\sim,
        \end{equation*} 
        where $\sim$ is an equivalence relation defined as follows: $(\tilde{\sigma}_1,\tilde{\sigma}_2)\sim(\tilde{\sigma}\uppri_1,\tilde{\sigma}\uppri_2)$ if and only if \myeqref{equ: equivalence relation between Inv times Inv} holds for some $\varphi_1$ and $\varphi_2$ isotopic to each other. 
    \end{thm}
    \par 
    The equivalence class in $\inv{S}\times\inv{S}\big/\sim$ corresponding to $\tilde{\sigma}_1$ and $\tilde{\sigma}_2$ in Proposition \ref{prop: equiv of Klein-bottled type} is denoted by $\llbracket\tilde{\sigma}_1,\tilde{\sigma}_2\rrbracket$. In general, we say the corresponding transnormal system $\F$ of Klein type is induced by $\tilde{\sigma}_1$ and $\tilde{\sigma}_2$.
    \begin{ex}\label{example: Klein-bottled type when 2-sphere}
        When DR-foils are homeomorphic to $\mS^2$, it is known that any fixed-point-free involution on sphere $\mS^2$ is conjugate to the antipodal map $\sigma_{\mS^2}$. Consequently, the set $\inv{\mS^2}\times\inv{\mS^2}\big/\sim$ consists a single element $\llbracket \sigma_{\mS^2},\sigma_{\mS^2}\rrbracket$. Then, 
        \begin{align*}
            \mathcal{K}_{\sigma_{\mS^2};\sigma_{\mS^2}}\simeq& \frac{[0,1]\times\mS^2}{(0,x)\sim(0,\sigma_{\mS^2}(x)),(1,x)\sim(1,\sigma_{\mS^2}(x))}\\
            \simeq & \frac{[0,\frac{1}{2}]\times\mS^2}{(0,x)\sim(0,\sigma_{\mS^2}(x))}\cup \frac{[\frac{1}{2},1]\times\mS^2}{(1,x)\sim(1,\sigma_{\mS^2}(x))}\\
            \simeq & \mR P^3\# \mR P^3.
        \end{align*}
    \end{ex}
    \begin{ex}
        Since $\mR P^3$ admits no fixed-point-free involution (see \cite{asoh1976classification}), there exists no transnormal system of Klein-bottled type whose DR-foils are homeomorphic to $\mR P^3$.
    \end{ex}
    \par    
    When $M$ is an orientable compact 3-manifold, the essential cover with respect to $\F$ is homeomorphic to $\mR\times S_g$ for some orientable closed surface $S_g$ of genus $g$, on which the foil-reflections $\sigma_1:(t,x)\mapsto(-t,\tilde{\sigma}_1(x))$ and $\sigma_2:(t,x)\mapsto(t_0-t,\tilde{\sigma}_2(x))$ preserve orientation. It follows that $\tilde{\sigma}_1$ and $\tilde{\sigma}_2$ are orientation-reversing fixed-point-free involutions on $S_g$. The set of such equivalence classes of transnormal systems is in one-to-one correspondence with  
    \begin{equation*}
        \invre{S_g}\times\invre{S_g}\big/\sim,
    \end{equation*}
    where $\invre{S_g}$ is the orientation-reversing subset of $\inv{S_g}$. Then, a conclusion similar to Proposition \ref{prop: CPC for toric type} holds for the Klein-bottled type.
    \begin{prop}\label{prop: CPC for Klein-bottled type}
        Let $M$ be an orientable compact 3-manifold, and let $\F$ be a transnormal system of Klein-bottled type on $M$ with DR-foils homeomorphic to the orientable closed surface $S_g$ of genus $g$. 
        \begin{enumerate}[(i)]
            \item If $g\le 1$, then, up to equivlence, $\F$ is a CPC transnormal system, and meanwhile $M$ is locally isometric to $\mS^2\times\mR$, $\mR^3$, Nil geometry or Sol geometry.
            \item If $g\ge 2$, $\F$ is equivalently a CPC transnormal system if and only if $\tilde{\sigma}_1$ and $\tilde{\sigma}_2$ are orientation-reversing fixed-point-free involutions of $S_g$ satisfying that $\tilde{\sigma}_2\circ\tilde{\sigma}_1$ is periodic, and the equivalence class of $\F$ corresponds to $\llbracket\tilde{\sigma}_1,\tilde{\sigma}_2\rrbracket\in\invre{S_g}\times\invre{S_g}\big/\sim$. In this case, up to equivlence of CPC transnormal systems, $M$ is locally isometric to $\mH^2\times\mR$.   
        \end{enumerate}
    \end{prop}
    \begin{proof}
        Assume that $\F$ is induced by $(\tilde{\sigma}_1,\tilde{\sigma}_2)\in \invre{S_g}\times\invre{S_g}$. 
        \par 
        If $g=0$, then $\F$ is equivalently the same as the one in Example \ref{example: Klein-bottled type when 2-sphere}. Recalling the CPC transnormal system of toric type with DR-foils homeomorphic to $\mS^2$ in the proof of Proposition \ref{prop: CPC for toric type} and considering $\bar{\sigma}:\mS^1\times\mS^2\to\mS^1\times\mS^2$ defined by $\bar{\sigma}(e^{i\theta},x)=(e^{-i\theta},\sigma_{\mS^2}(x))$, we obtain a CPC transnormal system of Klein-bottled type on $\mS^1\times\mS^2\big/\langle\bar{\sigma}\rangle$ by Corollary \ref{coro: Klein-bottled to toric}. In other words, $\mR P^3\# \mR P^3$ can be equipped with a Riemannian metric, locally isometric to the product space $\mS^2\times\mR$, such that it admits a CPC transnormal system of Klein-bottled type. 
        \par 
        If $g=1$, it follows from the classification of involutions (\cite{asoh1976classification}) that any orientation-reversing fixed-point-free involution on torus is conjugate to the unique involution $\sigma_{-}$ defined as 
        \begin{equation}\label{equ: sigma reverse orientation on torus}
            \sigma_{-}:\mR^2/\mZ^2\to\mR^2/\mZ^2,\quad [x,y]\mapsto [x+\frac{1}{2},-y].
        \end{equation}
        According to Proposition \ref{prop: equiv of Klein-bottled type}, it suffices to consider the cases $(\tilde{\sigma}_1,\tilde{\sigma}_2)=(h\circ\sigma_{-}\circ h^{-1},\sigma_{-})$, with $h(x,y)=(h_{11}x+h_{12}y,h_{21}x+h_{22}y)$ and $(h_{ij})\in\operatorname{SL}(2,\mZ)$. By Corollary \ref{coro: Klein-bottled to toric}, $(M,\F)$ is covered by $(\bar{M},\bar{\F})$ as a two-sheeted covering with $\bar{\F}$ being a transnormal system of toric type induced by $\tilde{\tau}=\tilde{\sigma}_2\circ\tilde{\sigma}_1$. Directly, 
        \begin{align*}
            \tilde{\tau}(x,y)=&\,\sigma_{-}\circ h\circ\sigma_{-}\circ h^{-1}(x,y)\\
            =&\,\bigl((h_{11}h_{22}+h_{12}h_{21})x-2h_{11}h_{12}y+\frac{1}{2}(h_{11}+1),-2h_{21}h_{22}x+(h_{11}h_{22}+h_{12}h_{21})y-\frac{1}{2}h_{21}\bigr). 
        \end{align*}
        In the following, we discuss each case separately, depending on whether $\tilde{\tau}$ is isotopically periodic, reducible, or Anosov, which is determined by the matrix 
        \[
            A=\left(
                \begin{array}{cc}
                    h_{11}h_{22}+h_{12}h_{21} & -2h_{11}h_{12}\\
                    -2h_{21}h_{22} & h_{11}h_{22}+h_{12}h_{21}
                \end{array}
            \right)\in \operatorname{SL}(2,\mZ). 
        \] 
        For each case, we verify that the corresponding foil-reflection is isometric under the Riemannian metric described in the proof of Proposition \ref{prop: CPC for toric type}. 
        \par 
        If $\tilde{\tau}$ is isotopic to a periodic homeomorphism, $A$ is periodic and direct computation yields $A=\pm I$. Then, the foil-reflection $\sigma_2(t,x,y)=(-t,\sigma_{-}(x,y))$ is obviously isometric under the corresponding metric. If $\tilde{\tau}$ is isotopic to a reducible homeomorphism, we find that 
        \[
            A=\pm \left(
                \begin{array}{cc}
                    1 & 2k\\
                    0 & 1
                \end{array}
            \right) \ \text{or}\ \pm \left(
                \begin{array}{cc}
                    1 & 0\\
                    2k & 1
                \end{array}
            \right)
        \]
        for some $k\in\mZ\backslash\{0\}$, and then $\sigma_2$ is also isometric under the corresponding metric. As for the case that $\tilde{\tau}$ is isotopic to an Anosov homeomorphism, through some calculations, we obtain $h_{11}h_{12}h_{21}h_{22}\neq 0$ and the eigenvectors of $A$ in \myeqref{equ: eigenvector of anosov} are
        \[
            \pp{\xi_1}=\sqrt{h_{11}h_{12}}\pp{x}+\sqrt{h_{21}h_{22}}\pp{y},\quad \pp{\xi_2}=\sqrt{h_{11}h_{12}}\pp{x}-\sqrt{h_{21}h_{22}}\pp{y}. 
        \]
        The metric in \myeqref{equ: metric of Anosov of torus} is exactly
        \[
            \metr=\ln^2\lambda\dd{t^2}+\lambda^{-2t}\left(\frac{\dd{x}}{2\sqrt{h_{11}h_{12}}}+\frac{\dd{y}}{2\sqrt{h_{21}h_{22}}}\right)+\lambda^{2t}\left(\frac{\dd{x}}{2\sqrt{h_{11}h_{12}}}-\frac{\dd{y}}{2\sqrt{h_{21}h_{22}}}\right),
        \]
        under which $\sigma_2$ is also isometric. Hence, we complete the proof for the case $g=1$. 
        \par 
        When $g\ge 2$, if $\tilde{\sigma}_2\circ\tilde{\sigma}_1$ is periodic, the group $G$ generated by $\tilde{\sigma}_1$ and $\tilde{\sigma}_2$ is finite. We obtain a $G$-invariant hyperbolic metric on $S_g$ and a CPC transnormal system of Klein-bottled type as in Proposition \ref{prop: CPC for toric type}. Conversely, if $\F$ is a CPC transnormal system of Klein-bottled type, the foils of the related transnormal system $\bar{\F}$ of toric type in Corollary \ref{coro: Klein-bottled to toric} also have constant principal curvatures, and $\bar{\F}$ is induced by $\tilde{\sigma}_2\circ\tilde{\sigma}_1$. The remainder of the proof follows exactly the same procedure as in the proof of Proposition \ref{prop: CPC for toric type}.
    \end{proof}

    \subsection{Transnormal systems of spherical type}\label{subsec: 3d and spherical type}
    We now suppose that $\F$ is a transnormal system of spherical type on Riemannian manifold $M$. We have known from Theorem \ref{thm: intro: four equivlent conditions} that it induces a linear double disk bundle decomposition $(\mathcal{D}_1\to L_1,\mathcal{D}_2\to L_2,\phi:\partial \mathcal{D}_1\to \partial \mathcal{D}_2)$ on $M$. Precisely, $\mathcal{P}_k:\mathcal{D}_k\to L_k$ is a linear $d_k$-disk bundle over $L_k$ for $k=1,2$, where $d_k$ is the codimension of $L_k$ in $M$; $M$ is obtained by attaching $\mathcal{D}_1$ and $\mathcal{D}_2$ along their common boundary by the diffeomorphism $\phi:\partial \mathcal{D}_1\to \partial \mathcal{D}_2$. Moreover, when $\F$ is of spherical type, both $d_1$ and $d_2$ are greater than one. Suppose that $\F\uppri$ is another transnormal system of spherical type on $M\uppri$ with each notation differing by a prime superscript. We provide a sufficient condition for determining the equivalence of two transnormal systems of spherical type. 
    \begin{prop}\label{prop: equiv of spherical type}
        If there exist bundle isomorphisms $\varphi_k$ between linear disk bundles $\mathcal{D}_k$ and $\mathcal{D}\uppri_k$ for $k=1,2$ such that $\phi\uppri$ is smoothly isotopic to $\left.\varphi_2\circ\phi\circ\varphi_1^{-1}\right|_{\partial\mathcal{D}\uppri_1}$, then $\F$ and $\F\uppri$ are equivalent transnormal systems. 
    \end{prop}
    \begin{proof}
        Since $\varphi_1:\mathcal{D}_k\to\mathcal{D}\uppri_k$ is a bundle isomorphism, there is a neighborhood $C$ (resp. $C\uppri$) of $\partial\mathcal{D}_1$ (resp. $\partial\mathcal{D}\uppri_1$) diffeomorphic to $\partial\mathcal{D}_1\times (1-\varepsilon,1]$ (resp. $\partial\mathcal{D}\uppri_1\times (1-\varepsilon,1]$), such that $\varphi_1$ restricted to $C$ maps $(r,y)\in C$ to $\big(r,\left.\varphi_1\right|_{\partial\mathcal{D}_1}(y)\big)\in C\uppri$. Without loss of generality, assume $\phi\uppri=\left.\varphi_2\circ\phi\circ\varphi_1^{-1}\right|_{\partial\mathcal{D}\uppri_1}\circ h^{-1}$ with $h:\partial\mathcal{D}\uppri_1\to \partial\mathcal{D}\uppri_1$ smoothly isotopic to identity. Define $\Phi_1:\mathcal{D}_1\to \mathcal{D}\uppri_1$ as
        \begin{equation*}
            \Phi_1(x)=\left\{
                \begin{aligned}
                    &\varphi_1(x),\quad& x\in \mathcal{D}_1-C,\\
                    &\left(r,\Psi_r(y)\right),\quad& x=(r,y)\in C,
                \end{aligned}
            \right. 
        \end{equation*}
        where $\{\Psi_r\}_{r\in[1-\varepsilon,1]}$ is the smooth isotopy between $\left.\varphi_1\right|_{\partial \mathcal{D}_1}$ and $h\circ \left.\varphi_1\right|_{\partial \mathcal{D}_1}$. Obviously, $\Phi_1$ maps foils to foils, and $\phi\uppri=\left.\varphi_2\circ\phi\circ\Phi_1^{-1}\right|_{\partial\mathcal{D}\uppri_1}$. Then, $(\Phi_1,\varphi_2)$ forms a diffeomorphism from $M$ to $M\uppri$, and the equivalence between $\F$ and $\F\uppri$ follows. 
    \end{proof}
    \par 
    When $M$ is a compact 3-manifold and $\F$ is of spherical type, each S-foil is either a single point or a circle. If the S-foil $L_1$ is a single point, $\mathcal{D}_1$ is diffeomorphic to a $3$-disk and the linear disk bundle $\mathcal{P}_1:\mathcal{D}_1\to L_1$ is unique up to linear bundle isomorphism. At that time, DR-foils are all diffeomorphic to the $2$-sphere. If $L_1$ is a circle, the linear disk bundle $\mathcal{P}_1:\mathcal{D}_1\to L_1$ have two possible structures up to linear bundle isomorphism: $\mathcal{D}_1$ is diffeomorphic to either a solid torus or a solid Klein bottle, depending on whether the bundle is trivial or twisted. In these cases, the DR-foils are all diffeomorphic to the torus or the Klein bottle, respectively. A similar discussion applies to the other S-foil $L_2$. Therefore, these two linear disk bundles must be isomorphic to each other. 
    \par 
    \paragraph{Case 1: $L_1$ and $L_2$ are both single points.} $M$, obtained by gluing two $3$-disks, is diffeomorphic to the $3$-sphere. It follows easily from Proposition \ref{prop: equiv of spherical type} that there exists a unique equivalence class of transnormal systems of spherical type. Up to equivalence, this transnormal system is induced by an $\operatorname{SO}(3,\mR)$ action on the standard sphere $\mS^3$. 
    \par 
    \paragraph{Case 2: $L_1$ and $L_2$ are both circles, while $\mathcal{D}_1$ and $\mathcal{D}_2$ are both solid tori.} Gluing $\mathcal{D}_1$ and $\mathcal{D}_2$ along their common boundary, we obtain the well-known lens space $L(p,q)$ for coprime $p,q\in\mZ$. 
    \par 
    \begin{rmk}\label{rmk: obtain lens space by pqst}
        In order to discuss precisely the equivalence classes of transnormal systems in this case, we briefly introduce the significance of $p$ and $q$ in the lens space $L(p,q)$ (see \cite[Chap.10]{martelli2016introduction} for example). Given two solid tori $\mathcal{D}_k=\mD^2\times\mS^1$ $(k=1,2)$, we can glue them along their boundaries by the attaching homeomorphism $\phi:\partial\mathcal{D}_1\to \partial\mathcal{D}_2$. From the theory on the mapping class group of the torus, we know that the isotopy class of $\phi$ is determined by the induced isomorphism $\phi\dwst$ between the homology groups $H_1(\partial\mathcal{D}_1;\mZ)$ and $H_1(\partial\mathcal{D}_2;\mZ)$. In detail, let $m_k=\partial\mD^2\times\{y\}$ and $l_k=\{x\}\times\mS^1$ represent the generators of $H_1(\partial\mathcal{D}_k;\mZ)$, and then $\phi\dwst$ satisfies 
        \begin{equation}\label{equ: relation to obtain Lens space}
            \left\{
                \begin{aligned}
                    \phi\dwst(m_{1})&=q\cdot m_{2}+p\cdot l_{2},\\
                    \phi\dwst(l_{1})&=s\cdot m_{2}+t\cdot l_{2},
                \end{aligned}
            \right.
            \quad \text{for some}\ \left(\begin{array}{cc}
                q&p\\
                s&t
            \end{array}\right)\in \operatorname{GL}(2,\mZ).
        \end{equation}
        Up to homeomorphism, the obtained manifold depends only on $p$ and $q$, and is denoted as lens space $L(p,q)$. An important property of lens spaces is that two lens spaces $L(p,q)$ and $L(p\uppri,q\uppri)$ are homeomorphic if and only if $p\uppri=\pm\,p$ and $q\uppri\equiv \pm\, q^{\pm 1}\hspace{-4pt}\mod p$. 
    \end{rmk}
    \par 
    \begin{prop}\label{prop: every lens has only one equiv}
        Assume $\F$ and $\F\uppri$ are transnormal systems of spherical type on $L(p,q)$ and $L(p\uppri,q\uppri)$, respectively, whose DR-foils are all diffeomorphic to the torus. Then, $\F$ is equivalent to $\F\uppri$ if and only if $L(p,q)$ is homeomorphic to $L(p\uppri,q\uppri)$. 
    \end{prop}
    \begin{proof}
        It suffices to show the `if' part. Assume, as detailed in Remark \ref{rmk: obtain lens space by pqst}, that the transnormal system $\F$ on $L(p,q)$ is constructed via an attaching map $\phi$ between the boundaries of two solid tori, and $\phi$ is isotopic to the linear homeomorphism of the torus represented by the matrix
        \begin{equation*}
            \left(\begin{array}{cc}
                q&p\\
                s&t
            \end{array}\right)\in \operatorname{GL}(2,\mZ). 
        \end{equation*}
        Similarly, the notations for $(L(p\uppri,q\uppri),\F\uppri)$ differ only by a prime superscript. Since $L(p,q)$ is homeomorphic to $L(p\uppri,q\uppri)$, we know $p\uppri=\pm\,p$ and $q\uppri\equiv \pm\, q^{\pm 1}\hspace{-4pt}\mod p$. Without loss of generality, suppose $qt-ps=q\uppri t\uppri-p\uppri s\uppri=1$. Then, we obtain
        \begin{equation}\label{equ: proof of spherical lens space 2}
            \left(\begin{array}{cc}
                q\uppri&p\uppri\\
                s\uppri&t\uppri
            \end{array}\right)=
            \left(\begin{array}{cc}
                \delta_1(q+n_1 p)&\delta_2p\\
                \delta_2(s+n_1t+n_2q+n_1n_2p)&\delta_1(t+n_2p)
            \end{array}\right),
        \end{equation}
        or
        \begin{equation}\label{equ: proof of spherical lens space 3}
            \left(\begin{array}{cc}
                q\uppri&p\uppri\\
                s\uppri&t\uppri
            \end{array}\right)=
            \left(\begin{array}{cc}
                \delta_1(t+n_2p)&\delta_2p\\
                \delta_2(s+n_1t+n_2q+n_1n_2p)&\delta_1(q+n_1 p)
            \end{array}\right),
        \end{equation}
        where $\delta_1,\delta_2\in\{1,-1\}$, $n_1,n_2\in\mZ$. On the other hand,  
        \begin{equation}\label{equ: isomorphism of solid torus}
            \varphi:\mD^2\times\mS^1\to\mD^2\times\mS^1,\quad (z_1,z_2)\mapsto(z_1z_2^{n_k},z_2^{\pm\, 1})\ \text{or}\ (\overline{z}_1z_2^{-n_k},z_2^{\pm\, 1})
        \end{equation}
        is an isomorphism of the disk bundle $\mD^2\times\mS^1$ over $\mS^1$, whose restriction to the boundary $\partial\mD^2\times\mS^1$ satisfies
        \[
            \left(\left.\varphi\right|_{\partial\mD^2\times\mS^1}\right)\dwst\left(
                \begin{array}{c}
                    m_k\\
                    l_k
                \end{array}
            \right)=
            \pm \left(
                \begin{array}{cc}
                    1&\\
                    n_k& \pm\,1
                \end{array}
            \right)\left(
                \begin{array}{c}
                    m_k\\
                    l_k
                \end{array}
            \right). 
        \]
        It follows directly that $\phi\uppri$ is isotopic to $\varphi_2\circ\phi\circ\left.\varphi_1^{-1}\right|_{\partial\mD^2\times\mS^1}$ for some appropriately chosen $n_k$. Thus, by Proposition \ref{prop: equiv of spherical type}, the proof is complete.
    \end{proof}
    \par 
    As established in \cite{mostert1957compact,neumann19683}, each lens space admits a cohomogeneity-one action by the Lie group $\mS^1\times\mS^1$. Immediately, we have
    \begin{coro}
        The transnormal system $\F$ of spherical type on $L(p,q)$ with DR-foils homeomorphic to the torus is, up to equivalence, induced by the action of the Lie group $\mS^1\times\mS^1$. 
    \end{coro}
    \par 
    \paragraph{Case 3: $L_1$ and $L_2$ are both circles, while $\mathcal{D}_1$ and $\mathcal{D}_2$ are both solid Klein bottles.} $M$ is obtained by gluing two solid Klein bottles along their common boundary via an attaching map $\phi:\mathbf{K}\to\mathbf{K}$, where $\mathbf{K}$ denotes the Klein bottle. The mapping class of the Klein bottle has been studied in \cite{lickorish1963homeomorphisms,lickorish1965homeomorphisms}, where it is shown that $\MCG{\mathbf{K}}\cong\mZ_2\times\mZ_2$. A standard model for $\mathbf{K}$ is given by  
    \begin{equation}\label{equ: def of Klein bottle}
        \mathbf{K}=\frac{[0,1]\times\mS^1}{(0,e^{i\theta})\sim(1,e^{-i\theta})},
    \end{equation} 
    while the solid Klein bottle, denoted by $\mathbf{SK}$, is represented as
    \begin{equation*}
        \mathbf{SK}=\frac{[0,1]\times\mD^2}{(0,r\cdot e^{i\theta})\sim(1,r\cdot e^{-i\theta})}.
    \end{equation*}
    The mapping class group $\MCG{\mathbf{K}}$ is generated by the homeomorphisms
    \[
        \phi_{1}(t,e^{i\theta})=(t,e^{i(\theta-2\pi t)})\quad\text{and}\quad \phi_2(t,e^{i\theta})=(1-t,e^{-i\theta}). 
    \]
    Here, $\phi_1$ is a non-trivial Dehn twist of the Klein bottle, while $\phi_2$ is the Y-homeomorphism (see \cite{lickorish1963homeomorphisms,lickorish1965homeomorphisms}). Both $\phi_1$ and $\phi_2$ extend naturally to homeomorphisms of $\mathbf{SK}$, where they act as bundle isomorphisms of $\mathbf{SK}$ viewed as a twisted disk bundle over the circle. By Proposition \ref{prop: equiv of spherical type}, the transnormal system $\F$ has a unique equivalence class in this case. Consequently, $M$ is obtained by the canonical gluing of two solid Klein bottles, yielding a space homeomorphic to the twisted $\mS^2$-bundle over the circle, i.e., 
    \[
        \mS^1\times_{\text{\tiny tw}} \mS^2= \frac{[0,1]\times\mS^2}{(0;x_1,x_2,x_3)\sim(1;-x_1,x_2,x_3)},
    \]
    where $(x_1,x_2,x_3)$ are the coordinates of $\mS^2$ in $\mR^3$. 
    \par $\ $\par 
    In summary, Table \ref{table: Classification of spherical type in closed 3d} lists all the equivalence classes of transnormal systems of spherical type on compact 3-manifolds. 
    \begin{table}[!htbp]
        \centering \small
        \caption{Equivalence classes of transnormal systems of spherical type on closed 3-manifolds.}
        \label{table: Classification of spherical type in closed 3d}
        \renewcommand\arraystretch{1.3}
        \tabcolsep=0.5cm
        \scalebox{1}{
        \begin{tabular}{C{2cm}|C{2cm}|C{1.5cm}|C{2.5cm}|C{2cm}} 
            \toprule[1pt]
            Manifold $M$  & DR-foils & S-foils & Cohomogeneity-One Action & Orientable ($M$) \\
            \cline{1-5}\cline{1-5}
            $\mS^3$ & $2$-sphere & Point & $\operatorname{SO}(3,\mR)$ & \multirow{2}{*}{Yes}\\
            \cline{1-4}
            Lens Space & Torus & \multirow{2}{*}{Circle} & $\mS^1\times\mS^1$ & \\
            \cline{1-2}\cline{4-5}
            $\mS^1\times_{\text{\tiny tw}} \mS^2$ & Klein bottle && None & No \\
            \bottomrule[1pt]
        \end{tabular}
        }
    \end{table}
    And each equivalence class has a good representative:
    \begin{prop}\label{prop: spherical type equiv CPC}
        Let $\F$ be a transnormal system of spherical type on a compact Riemannian 3-manifold $M$. Up to equivalence, $\F$ is a CPC transnormal system, and meanwhile $M$ is locally isometric to either the standard $\mS^3$ or $\mS^2\times\mR$. 
    \end{prop} 
    
    \subsection{Transnormal systems of real-projective type}\label{subsec: 3d and real-projective type}
    Let $\F=\{F_t\}_{t\in[0,\frac{t_0}{2}]}$ be a transnormal system of real-projective type on Riemannian manifold $M$. Recalling Proposition \ref{prop: non-essential to essential} (iv), $(\tilde{M},\tilde{\F}=\{\tilde{F}_t\}_{t\in[0,t_0]})\to (M,\F)$ is a two-sheeted covering map, where $\tilde{\F}$ is of spherical type and $M=\tilde{M}\big/\langle\sigma\rangle$ with $\sigma(\tilde{F}_t)=\tilde{F}_{t_0-t}$ being an isometric foil-reflection. For simplicity, denote $t_1=\frac{t_0}{2}$ and $\sigma_{t_1}$ as the restriction of $\sigma$ to $\tilde{F}_{t_1}$. Immediately, we have  
    \begin{equation}\label{equ: M in real-projective type}
        M= \tilde{M}/\langle\sigma\rangle\simeq \frac{\mathcal{D}}{x\sim \sigma_{t_1}(x)\in\partial\mathcal{D}},
    \end{equation}
    where $\mathcal{D}$ is the submanifold of $\tilde{M}$ composed of all foils $\tilde{F}_t$ with $t\le t_1$. Moreover, $\mathcal{D}$ is the linear disk bundle over the unique S-foil in $\F$, and its boundary $\partial\mathcal{D}$ is diffeomorphic to any DR-foil of $\F$. Note that $\sigma_{t_1}$ is a fixed-point-free involution on $\tilde{F}_{t_1}$, that is, $\sigma_{t_1}\in\inv{\partial\mathcal{D}}$. Suppose that $\F\uppri$ is another transnormal system of real-projective type on $M\uppri$, with each notation differing by a prime superscript. A sufficient criterion for equivalence of $\F$ and $\F\uppri$ is as follows. 
    \begin{prop}\label{prop: equiv of real-projective type}
        If there exist a bundle isomorphism $\varphi$ between linear disk bundles $\mathcal{D}$ and $\mathcal{D}\uppri$ such that $\sigma\uppri_{t\uppri_1}=\psi\circ\sigma_{t_1}\circ\psi^{-1}$ with $\psi:\partial\mathcal{D}\to\partial\mathcal{D}\uppri$ smoothly isotopic to $\left.\varphi\right|_{\partial\mathcal{D}}$, then $\F$ and $\F\uppri$ are equivalent transnormal systems. 
    \end{prop}
    \par 
    The proof follows from Proposition \ref{prop: non-essential to essential} (iv) and the argument in Proposition \ref{prop: equiv of spherical type}.
    \par 
    When $\F$ is a transnormal system of real-projective type on a compact 3-manifold $M$, the equivalence class of $\F$ is determined by the disk bundle $\mathcal{D}\to L$ and the conjugacy class of $\sigma_{t_1}$ via an isotopy to the identity. As in Subsection \ref{subsec: 3d and spherical type}, $\mathcal{D}$ has only three possibilities up to bundle isomorphism.
    \par 
    \paragraph{Case 1: $L$ is a single point.} In this case, $\mathcal{D}$ is a 3-disk, and $\sigma_{t_1}\in\inv{\mS^2}$. Any fixed-point-free involution on $\mS^2$ is conjugate to the antipodal map $\sigma_{\mS^2}$ via an isotopy to the identify. By Proposition \ref{prop: equiv of real-projective type}, there is a unique equivalence class of transnormal systems of real-projective type whose S-foil is a single point. Hence, $M$ is diffeomorphic to $\mR P^3$, and $\F$, up to equivalence, is induced by the cohomogeneity-one action of $\operatorname{SO}(3,\mR)$. 
    \par 
    \noindent
    \paragraph{Case 2: $L$ is a circle, while $\mathcal{D}$ is a solid torus.} Then, $\sigma_{t_1}\in\inv{\partial\mD^2\times\mS^1}$. By the classification of involutions on surfaces (see \cite{asoh1976classification}), any fixed-point-free involution on the torus is conjugate to either
    \[
        \sigma_{+}:\partial\mD^2\times\mS^1\to \partial\mD^2\times\mS^1,\quad (z_1,z_2)\mapsto (-z_1,z_2),
    \]
    or $\sigma_{-}$, depending on whether it preserves orientation or not, where $\sigma_{-}(z_1,z_2)=(-z_1,\bar{z}_2)$ is defined in \myeqref{equ: sigma reverse orientation on torus} (in different coordinates).
    Hence, there exists a homeomorphism $h$ on the torus such that 
    \[
        \sigma_{t_1}=h\circ\sigma_{+}\circ h^{-1} \quad \textrm{or}\quad \sigma_{t_1}=h\circ\sigma_{-}\circ h^{-1}.
    \]
    Moreover, by Proposition \ref{prop: equiv of real-projective type}, only the isotopy class of $h$ is relevant, so we may assume 
    \[ 
        h(z_1,z_2)=(z_1^{h_{11}}z_2^{h_{12}},z_1^{h_{21}}z_2^{h_{22}})
    \] 
    with $(h_{ij})\in\operatorname{SL}(2,\mZ)$.
    \par 
    If $\sigma_{t_1}=h\circ\sigma_{+}\circ h^{-1}$, then the SR-foil of $\F$ is diffeomorphic to the torus. A straightforward calculation yields
    \[
        \sigma_{t_1}(z_1,z_2)=((-1)^{h_{11}}z_1,(-1)^{h_{21}}z_2).
    \]
    Since $h_{11}$ and $h_{21}$ cannot both be even, there are exactly two subcases:
    \begin{enumerate}[(i)]
        \item $h_{21}$ is even, i.e., $\sigma_{t_1}(z_1,z_2)=(-z_1,z_2)$. Recalling \myeqref{equ: M in real-projective type}, the map $\sigma_{t_1}$ identifies antipodal points on the boundary of the disk fibers of $\mathcal{D}$, implying that $M$ is an $\mR P^2$-bundle over the circle. Since such a bundle must be trivial, $M$ is diffeomorphic to $\mS^1\times\mR P^2$. Up to equivalence, $\F$ can be induced by the cohomogeneity-one action of $\mS^1\times\mS^1$ on $\mS^1\times\mR P^2$. 
        \item $h_{21}$ is odd. Considering the bundle isomorphism in \myeqref{equ: isomorphism of solid torus} and applying Proposition \ref{prop: equiv of real-projective type}, we deduce that $\sigma_{t_1}$ and $\psi\circ\sigma_{t_1}\circ\psi^{-1}$ correspond to the same equivalence class of $\F$, where $\psi=\left.\varphi\right|_{\partial\mD^2\times\mS^1}$ and
        \begin{equation}\label{equ: 3-d real-projective type relation in equiv class}
            \psi\circ\sigma_{t_1}\circ\psi^{-1}(z_1,z_2)=((-1)^{h_{11}+nh_{21}}z_1,(-1)^{h_{21}}z_2), \quad n\in\mZ.
        \end{equation}
        By selecting an appropriate $n\in\mZ$ in \myeqref{equ: 3-d real-projective type relation in equiv class}, we ensure that $h_{11}+nh_{21}$ is even. Consequently, we may assume $\sigma_{t_1}(z_1,z_2)=(z_1,-z_2)$. The map $\sigma_{t_1}$ identifies the disk fibers of the disk bundle at antipodal positions. Thus, $M$ remains a fiber bundle, with fibers diffeomorphic to $\mS^2$, and base space given by the original base $\mS^1$ modulo the antipodal map. Since $M$ is non-orientable, it must be the twisted bundle $\mS^1\times_{\text{\tiny tw}} \mS^2$. Up to equivalence, $\F$ can be induced by the cohomogeneity-one action of $\mS^1\times\mS^1$ on $\mS^1\times_{\text{\tiny tw}} \mS^2$. 
    \end{enumerate}
    \par 
    If $\sigma_{t_1}=h\circ\sigma_{-}\circ h^{-1}$, the SR-foil of $\F$ is diffeomorphic to the Klein bottle. A straightforward calculation yields
    \begin{equation}\label{equ: involution of orientation-reversing class}
        \sigma_{t_1}(z_1,z_2)=\left((-1)^{h_{11}}z^{h_{11}h_{22}+h_{12}h_{21}}_1z^{-2h_{11}h_{12}}_2,(-1)^{h_{21}}z^{2h_{21}h_{22}}_1z^{-h_{11}h_{22}-h_{12}h_{21}}_2\right).
    \end{equation}
    The resulting manifold $M$ is a Klein space defined in \cite{kim1978some,kim1981involutions} as below (also defined equivalently in terms of Seifert fibre spaces in \cite{rubinstein19793}):
    \begin{defi}\label{def: Klein space}
        Under complex coordinates, consider $\sigma_{-}(z_1,z_2)=(-z_1,\bar{z}_2)$ and $f(z_1,z_2)=(z_1^rz_2^s,z_1^pz_2^q)$ on $\mS^1\times\mS^1$, where $p,q,r,s$ are integers satisfying $ps-qr=1$. By gluing $(0,z_1,z_2)$ with $(0,\sigma_{-}(z_1,z_2))$ in $[0,1]\times\mS^1\times\mS^1$, and then gluing $(1,z_1,z_2)$ to the corresponding point $f(z_1,z_2)$ on the boundary of $\mD^2\times\mS^1$, the resulting topological space forms a compact, boundaryless manifold, denoted by $M(p,q,r,s)$, and called a {\bf Klein space}. 
    \end{defi}
    \par 
    \begin{prop}[\cite{kim1978some}]
        Two Klein spaces $M(p,q,r,s)$ and $M(p\uppri,q\uppri,r\uppri,s\uppri)$ are homeomorphic if and only if $p=\pm p\uppri,q=\pm q\uppri$. Therefore, $M(p,q,r,s)$ can be simply written as $M(p,q)$, where $p$ and $q$ are coprime. 
    \end{prop}
    \par 
    A direct analysis shows that $M$ is precisely the Klein space $M(h_{21},h_{22})$. 
    Moreover, Klein spaces exhibit structural similarities with lens spaces; for further details, see \cite{kim1978some,kim1981involutions}. In this setting, the properties of $\F$ on Klein spaces parallel those in Proposition \ref{prop: every lens has only one equiv}:
    \begin{prop}
        Assume $\F$ and $\F\uppri$ are transnormal systems of real-projective type on Klein spaces $M(h_{21},h_{22})$ and $M(h_{21}\uppri,h_{22}\uppri)$, respectively, whose DR-foils are all diffeomorphic to the torus and SR-foils are diffeomorphic to the Klein bottle. Then, $\F$ is equivalent to $\F\uppri$ if and only if $M(h_{21},h_{22})$ is homeomorphic to $M(h_{21}\uppri,h_{22}\uppri)$. 
    \end{prop}
    \begin{proof}
        Assume that $\F$ on $M(h_{21},h_{22})$ is constructed via an attaching map $\sigma_{t_1}=h\circ\sigma_{-}\circ h^{-1}$ on the boundary of the disk bundle $\mD^2\times\mS^1$, and similarly, $\F\uppri$ on $(M(h_{21}\uppri,h_{22}\uppri),\F\uppri)$ is defined in the same way, with corresponding prime superscripts. The relation between $(h_{ij})$ and $(h_{ij}\uppri)$ is given by
        \[
            \left(
                \begin{array}{cc}
                    h\uppri_{11}& h\uppri_{12}\\
                    h\uppri_{21}& h\uppri_{22}
                \end{array}
            \right)=\pm\left(
                \begin{array}{cc}
                    h_{11}+n\,h_{21}& h_{12}+n\,h_{22}\\
                    h_{21}& h_{22}
                \end{array}
            \right)
        \]
        or 
        \[
            \left(
                \begin{array}{cc}
                    h\uppri_{11}& h\uppri_{12}\\
                    h\uppri_{21}& h\uppri_{22}
                \end{array}
            \right)=\pm\left(
                \begin{array}{cc}
                    -(h_{11}+n\,h_{21})& h_{12}+n\,h_{22}\\
                    h_{21}& -h_{22}
                \end{array}
            \right),
        \]
        for some $n\in \mZ$. As in the proof of Proposition \ref{prop: every lens has only one equiv}, applying the criterion given in Proposition \ref{prop: equiv of real-projective type} and considering the bundle isomorphisms in \myeqref{equ: isomorphism of solid torus} establishes the result.
    \end{proof}
    \par 
    Among Klein spaces, there are two special cases, $M(0,1)$ and $M(1,0)$, which are diffeomorphic to $\mR P^3\#\mR P^3$ and $\mS^2\times\mS^1$, respectively. All other Klein spaces $M(p,q)$ have finite fundamental groups, and their universal covering space is $\mS^3$ (see \cite{kim1978some}). In summary, all Klein spaces can be equipped with a standard Riemannian metric that is locally isometric to either $\mS^2\times\mR$ or $\mS^3$. Within each equivalence class, there exists a corresponding $\F$ under this standard metric, which is a CPC transnormal system. However, according to the classification of cohomogeneity-one actions on compact 3-manifolds (\cite{mostert1957compact,neumann19683}), no $\F$ in this context can be induced by a cohomogeneity-one Lie group action.
    \par 
    \noindent
    \paragraph{Case 3: $L$ is a circle, while $\mathcal{D}$ is a solid Klein bottle.} In this case, $\mathcal{D}$ is a twisted 2-disk bundle over the circle, and $\sigma_{t_1}\in\inv{\mathbf{K}}$. According to \cite{asoh1976classification}, any fixed-point-free involution $\sigma_{t_1}$ on the Klein bottle can be written as $h\circ\sigma_{\mathbf{K}}\circ h^{-1}$, where $h$ is a homeomorphism of the Klein bottle and $\sigma_{\mathbf{K}}(t,e^{i\theta})=(t,-e^{i\theta})$ under \myeqref{equ: def of Klein bottle}. As in Subsection \ref{subsec: 3d and spherical type}, within each isotopy class of $h$, there exists a homeomorphism that extends naturally to a bundle isomorphism on the twisted disk bundle $\mathbf{SK}$. By Proposition \ref{prop: equiv of real-projective type}, the transnormal system $\F$ has a unique equivalence class in this case. Here, $M$ is obtained by gluing the boundary of $\mathbf{SK}$ along $\sigma_{\mathbf{K}}$. Since $\sigma_{\mathbf{K}}$ identifies antipodal points on the boundary of the disk fibers of $\mathbf{SK}$, $M$ is an $\mR P^2$-bundle over the circle, i.e., $\mS^1\times\mR P^2$. Both the DR-foils and SR-foil of $\F$ are diffeomorphic to the Klein bottle, and $\F$ is not equivalent to any transnormal system induced from a cohomogeneity-one Lie group action. 
    \par $\ $\par 
    In summary, Table \ref{table: Classification of real-projective type in closed 3d} lists all equivalence classes of transnormal systems of real-projective type on compact 3-manifolds. 
    \begin{table}[!htbp]
        \centering \footnotesize
        \caption{Equivalence classes of transnormal systems of real-projective type on closed 3-manifolds. }
        \label{table: Classification of real-projective type in closed 3d}
        \renewcommand\arraystretch{1.3}
        \tabcolsep=0.3cm
        \begin{tabular}{C{2cm}|C{1.8cm}|C{1.8cm}|C{2cm}|C{2.3cm}|C{1.5cm}}
            \toprule[1pt]
            Manifold $M$ & DR-foils & SR-foil & Essential Cover w.r.t. $\F$ & Cohomogeneity-One Action & Orientable ($M$) \\
            \hline
            $\mR P^3$  & 2-sphere & $\mR P^2$ & $\mS^3$ & $\operatorname{SO}(3,\mR)$ & Yes\\
            \hline
            $\mS^1\times\mR P^2$  & \multirow{3}{*}{Torus} & \multirow{2}{*}{Torus} & \multirow{2}{*}{$\mS^1\times\mS^2$} & \multirow{2}{*}{$\mS^1\times\mS^1$}& \multirow{2}{*}{No}\\
            \cline{1-1}
            $\mS^1\times_{\text{\tiny tw}} \mS^2$ & & & & &\\
            \cline{1-1}\cline{3-6}
            Klein Spaces & & \multirow{2}{*}{Klein bottle}  & Lens spaces & \multirow{2}{*}{None} & Yes\\
            \cline{1-2}\cline{4-4}\cline{6-6}
            $\mS^1\times\mR P^2$ & Klein bottle & & $\mS^1\times_{\text{\tiny tw}} \mS^2$ & & No\\
            \bottomrule[1pt]
        \end{tabular}
    \end{table}
    Based on the discussion in this subsection, although some of these equivalence classes are not induced by a cohomogeneity-one action, a conclusion similar to Proposition \ref{prop: spherical type equiv CPC} still holds:
    \begin{prop}\label{prop: real-projective type equiv CPC}
        Let $\F$ be a transnormal system of real-projective type on a compact Riemannian 3-manifold $M$. Up to equivalence, $\F$ is a CPC transnormal system, and meanwhile $M$ is locally isometric to either the standard $\mS^3$ or $\mS^2\times\mR$. 
    \end{prop}

    \subsection{Conclusions}
    Based on the classifications and related results from Subsections \ref{subsec: 3d and toric type} to \ref{subsec: 3d and real-projective type}, we can immediately prove the main theorems concerning CPC transnormal systems. By combining Propositions \ref{prop: CPC for toric type}, \ref{prop: CPC for Klein-bottled type}, \ref{prop: spherical type equiv CPC}, and \ref{prop: real-projective type equiv CPC}, we complete the proof of Theorem \ref{thm: intro: CPC has standard}. Finally, Theorem \ref{thm: intro: no CPC in some equivalence class} follows from Propositions \ref{prop: CPC for toric type} and \ref{prop: CPC for Klein-bottled type}, or alternatively, directly from the next example. 
    \begin{ex}\label{example: mapping torus of pseudo-anosov of hyperbolic surface}
        Let $S_g$ be an orientable closed surface of genus $g\ge 2$, and suppose $\tau\in\homeoori{S_g}$ is a pseudo-Anosov homeomorphism of $S_g$. According to the Nielsen–Thurston classification, $\tau$ is not isotopic to any periodic homeomorphism, and the mapping torus $\mathcal{M}_{\tau}$ admits a hyperbolic structure, that is, it has a standard metric locally isometric to $\mH^3$ (see, for example, \cite{thurston1986hyperbolic}). Immediately, the transnormal system $\F$ of toric type induced by $\tau$ is not equivalent to any CPC transnormal system by Proposition \ref{prop: CPC for toric type}. Assume $\metr$ is the Riemannian metric on $\mathcal{M}_{\tau}$ in Theorem \ref{thm: intro: CPC has standard}, such that $\F$ is induced from an isoparametric function. It is then claimed that $\metr$ can never become the standard metric locally isometric to $\mH^3$, as explained below: 
        If this were the case, the lifting transnormal system on the universal Riemannian cover $\mH^3$ would also be induced by an isoparametric function. This would imply that it is, in fact, a CPC transnormal system, according to the classification of isoparametric functions on $\mH^3$. Consequently, $\F$ itself would also be a CPC transnormal system, leading to a contradiction. 
        \par    
        In summary, this is a significant example in the study of cohomogeneity-one actions, CPC transnormal systems, isoparametric foliations, and transnormal systems. This manifold $\mathcal{M}_{\tau}$ does not admit both a cohomogeneity-one action and a CPC transnormal system for any Riemannian metric. However, it may admit an isoparametric foliation or a hyperbolic metric, but not both simultaneously. 
    \end{ex}
    \begin{rmk}\label{remark: why lift directly to universal is bad}
        In Section \ref{sec: covering and lifting}, we established that a transnormal system on a Riemannian manifold can be lifted to a transnormal system on its universal Riemannian cover. However, this does not imply that our discussion should primarily focus on transnormal systems in simply connected manifolds. On the contrary, in some cases, lifting to the universal cover introduces additional complexity. This is because our classification is carried out under the equivalence, which is not necessarily preserved under coverings. Specifically, suppose $\F$ is lifted to $\hat{\F}$. If $\hat{\F}\uppri$ is equivalent to $\hat{\F}$, then in general, there does not exist another transnormal system equivalent to $\F$ that can be lifted to $\hat{\F}\uppri$. Consequently, reducing connectivity via lifting prevents classification under the equivalence in Definition \ref{def: equivalence of tran sys}, whereas the latter often provides a more effective simplification than the former.
    \end{rmk}

    \newpage
    \bibliographystyle{alpha}
    \bibliography{isoparametric}

\end{document}